\RequirePackage{fix-cm}

\RequirePackage{amsmath}
\documentclass[twocolumn]{svjour3}          %
\smartqed  %
\usepackage{graphicx}

\usepackage{amsthm}
\usepackage{amssymb}
\usepackage[all,cmtip]{xy}
\usepackage{enumerate}
\usepackage{algorithmic}
\usepackage{algorithm}
\usepackage{multirow}
\usepackage{mathtools}
\mathtoolsset{showonlyrefs}

\usepackage{microtype}
\usepackage{adjustbox}
\usepackage[title]{appendix}

\newcommand{\RR}{\mathbb{R}}
\DeclareMathOperator{\tr}{tr}
\newcommand{\norm}[1]{\left\Vert#1\right\Vert}

\let\on=\operatorname
\newcommand{\Imm}{\on{Imm}}
\newcommand{\Diff}{\on{Diff}}
\newcommand{\degree}{\operatorname{deg}}
\newcommand{\degbar}{\overline{\operatorname{deg}}}

\newtheorem{thm}{Theorem}

\newtheorem{rem}{Remark}
\newtheorem{prop}{Proposition}

\usepackage{todonotes}

\begin{document} \sloppy

\title{Shape Analysis of Surfaces Using General Elastic Metrics%
\thanks{M. Bauer was partially supported by NSF-grant 1912037 (collaborative research in connection with NSF-grant 1912030). S. C. Preston was partially supported by Simons Foundation Collaboration Grant for Mathematicians no. 318969. E. Klassen was partially supported by Simons Foundation Collaboration Grant for Mathematicians no. 317865.}
}

\author{Zhe Su \and Martin Bauer  \and Stephen C. Preston \and Hamid Laga  \and Eric Klassen %
}

\institute{
           Z. Su \at
           Department of Mathematics, Florida State University, Tallahassee, USA. \\
           \email{zsu@math.fsu.edu}
          \and
M. Bauer \at
              Department of Mathematics, Florida State University, Tallahassee, USA. \\
              \email{bauer@math.fsu.edu}           %
           \and
           S. Preston \at
           Department of Mathematics, Brooklyn College and CUNY Graduate Center, New York, USA. \\
           \email{stephen.preston@brooklyn.cuny.edu}
                       \and
           H. Laga  \at
          Information Technology, Mathematics and Statistics, Murdoch University, Australia and with the Phenomics and Bioinformatics Research Centre, University of South Australia, Adelaide, Australia.\\
           \email{H.Laga@murdoch.edu.au}
           \and
           E. Klassen \at
           Department of Mathematics, Florida State University, Tallahassee, USA. \\
           \email{klassen@math.fsu.edu}
}

\date{Received: date / Accepted: date}

\maketitle

\begin{abstract}
In this article we introduce a family of elastic metrics on the space of parametrized surfaces in 3D space using a corresponding family of metrics on the space of vector valued one-forms. We provide a numerical framework for the computation of geodesics with respect to these metrics.
The family of metrics is invariant under rigid motions and reparametrizations; hence it induces a metric on the ``shape space" of surfaces. This new class of metrics
generalizes a previously studied family of elastic metrics and includes in particular the Square Root Normal Function (SRNF) metric, which has been proven successful in various applications. We demonstrate our framework by showing several examples of geodesics  and compare our results with earlier results obtained from the SRNF framework.
\keywords{Shape spaces \and Vector valued one-forms \and Elastic metrics \and SRNF metric \and Surface registration}
\end{abstract}

\section{Introduction}
Shape analysis of surfaces in $\mathbb R^3$ has been motivated by many applications in bioinformatics, computer graphics and medical imaging, see e.g., \cite{kilian2007geometric,kurte2011Anatomical,brett2003Human,grenandery1998Anatomy,heeren2012time,tumpach2016gauge}. In most applications the actual parametrization of the surfaces under consideration is unknown and one is only able to observe the  ``shape'' of the object, i.e., a priori the point correspondences between the surfaces are unknown and should be an output of the  performed analysis. Furthermore, we will often identify surfaces that only differ by a rigid motion.
Thus, we define the shape space of surfaces as the quotient space of all parametrized surfaces  modulo  the group of reparametrizations and/or the group of rigid motions. One goal in shape analysis is to quantify the differences and find the optimal deformations between the given objects; see Figure~\ref{fig.geodesic.introduction} for two examples of  optimal deformations between distinct surfaces.

\begin{figure*}[ht]
	\centering
	\includegraphics[scale=0.45]{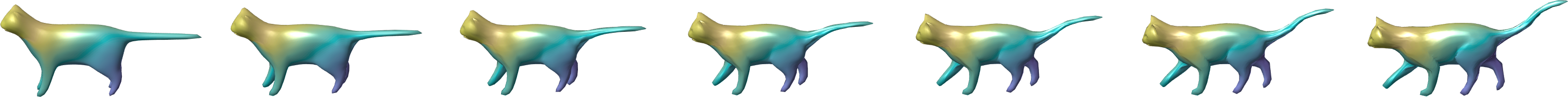}
	\vskip.20in
	\includegraphics[scale=0.45]{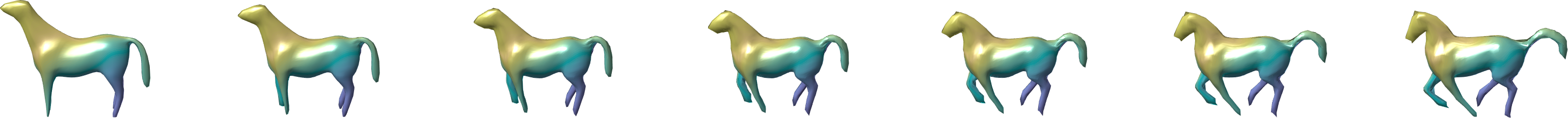}
	\caption{Geodesics between shapes in the space of unparametrized surfaces $\Imm(S^2, \RR^3)/\Diff_+(S^2)$ with respect to the split metric \eqref{eq.metric_abcd} for a choice of coefficients $(1,1,0.1,0)$.}
	\label{fig.geodesic.introduction}	
\end{figure*}

The main challenge in the context of shape analysis of surfaces consists in the registration problem, i.e., finding the (optimal) point correspondences between distinct surfaces, which can then be used as the basis for the resulting statistical analysis. In previous work, the correspondence problem has often been solved in a preprocessing step, which is then followed by an independent statistical analysis of the resulting parametrized surfaces. This approach can yield several undesirable consequences on the statistical analysis, see e.g.,~\cite{srivastava2016functional}.

The goal of elastic shape analysis is to formulate this problem in a unified framework: using a \emph{reparametrization invariant} metric on the space of all parametrized surfaces  one then studies the induced Riemannian metric on the quotient space. Using this approach, the point correspondences and the resulting statistical analysis can be performed in a consistent way.

In the past years several metrics and frameworks have been proposed as potential approaches to this goal, see e.g., \cite{jermyn2012SRNF,laga2017numerical,kurtek2013landmark,bauer2014overview,srivastava2016functional,tumpach2016gauge}.
In particular, a class of elastic metrics has been proposed in \cite{jermyn2017}, which is defined as a weighted sum of three components that measure the differences in shearing, stretching and bending of the surface. This family of metrics is actually a subfamily of the general class of reparameterization invariant Sobolev metrics, as studied in \cite{bauer2011sobolev,bauer2012sobolev,bauer2014overview}. It is also a natural generalization of the family of elastic $G^{a,b}$-metrics on the space of curves \cite{mio2007Elastic}, which has been proven efficient and successful in numerous  shape analysis applications \cite{srivastava2011Shape,younes1998computable,younes2008metric,srivastava2016functional,SuKuKlSr2014,ZhSuKlLeSr2015,CeEsSch2015,zhe2018Homogeneous}.

To obtain a numerically efficient representation,  Srivastava et al.  \cite{srivastava2011Shape} introduced the so-called Square Root Velocity Function (SRVF) for comparing curves. In this framework, the space of curves endowed with the elastic metric for a particular choice of coefficients is isometric to an $L^2$-space, which makes the computation of geodesics extremely easy and efficient. Motivated by this progress, Jermyn et al. \cite{jermyn2012SRNF} introduced the Square Root Normal Function (SRNF) representation for elastic shape analysis of surfaces and  showed that the $L^2$-metric on the space of SRNFs corresponds to one member of a more general class of elastic metrics on the space of surfaces. While it is computationally efficient, there are several drawbacks 
to this approach: the SRNF metric only consists of the last two terms of the general elastic metric for surfaces and is thus highly degenerate; i.e., there exists a high-dimensional space of deformations that has no cost in this framework\footnote{See the article \cite{klassenmichor2019} for an example of a path of closed surfaces that connects two distinct shapes, such that the whole path has the same SRNF.}. Furthermore, the SRNF map is neither injective nor surjective, and  its  image is not fully understood. In consequence there exists no analytic formula for geodesics in the image space and 
geodesics are usually approximated by numerically  inverting the straight line between the given SRNFs, where each inversion is calculated as the solution to an optimization problem~\cite{laga2017numerical}.

\noindent{\bf Contributions of this article:} The purpose of the present article is to introduce a numerical framework for the computation of the geodesic initial and boundary value problem with respect to a family of metrics that contains the general elastic metric as a special case. The framework complements \cite{bauer2018OneForms} which defined, using vector valued one-forms, a metric on the space of surfaces that is invariant under rigid motions and reparametrizations. It does not require a numerical inversion of the SRNF-map and thus overcomes some of the aforementioned difficulties. Furthermore, this framework will allow us in the future to choose the constants of the metric in a data driven way, which has potential importance in many applications.  See~\cite{kurtek2018simplifying,bauer2018relaxed} for related considerations regarding the choice of constants for the elastic metric on the space of curves.

\noindent{\bf Acknowledgements:} The authors thank Anuj Srivastava and all the members of the Florida State statistical shape analysis group for helpful discussions during the preparation of this manuscript. In addition we are grateful to Sebastian Kurtek and Barbara Tumpach for discussion about the implementation of the minimization over the diffeomorphism group. 

\section{Mathematical Framework and Background}
In this section we will give the formal definition of the space of shapes and describe the general elastic metric. Then we will introduce a new representation for the elastic metric using vector valued one-forms, which will still allow us to obtain an efficient discretization of the geodesic boundary value problem.

From here on we will model a surface as an immersion $f$ from a model space $M$ into $\RR^3$, i.e., a smooth map from $M$ to $\RR^3$ that has an injective tangent mapping. Here $M$ is a two-dimensional compact manifold encoding the topology of the objects under consideration. Typically, choices of $M$ include the two-sphere $M=S^2$ or the sheet $M=[0,1]^2$.

Denote by $\Imm(M, \RR^3)$ the space of all immersions.
To define the space of shapes, we now consider the actions of the group of rigid motions and the group of diffeomorphisms on $\Imm(M, \RR^3)$. The group of rigid motions is given by the semidirect product of the group of rotations and the group of translations, i.e., $\on{SO}(3)\ltimes\RR^3$, where $\on{SO}(3)$ is the set of all rotation matrices. It acts on $\Imm(M, \RR^3)$ as follows:
\begin{align}
\left(\on{SO}(3)\ltimes\RR^3\right)\times \on{Imm}(M, \RR^3)&\to \on{Imm}(M, \RR^3)\\
\left((R,v), f\right)&\mapsto Rf+v.
\end{align}
Denote by $\Diff_+(M)$ the group of diffeomorphisms that preserve the orientation of $M$. The action of $\on{Diff}_+(M)$ on $\on{Imm}(M, \RR^3)$ is given by composition from the right:
\begin{align}
\on{Imm}(M, \RR^3)\times\Diff_+(M)&\to\on{Imm}(M, \RR^3)\\
(f, \gamma)&\mapsto f\circ\gamma.
\end{align}
We say that two immersions $f_1$ and $f_2$ have the {\it same shape} if they are in the same orbit of the action of $\Diff_+(M)$, or both actions depending on whether we want to mod out rigid motions. The space of shapes can then be defined as the quotient space:
\begin{align}
\mathcal{S}(M,\RR^3) = \on{Imm}(M, \RR^3)/\mathcal{G},
\end{align}
where $\mathcal{G} = \Diff_+(M)$ or $\mathcal{G} = \on{Diff}_+(M)\times\on({SO}(3)\ltimes\RR^3) $.

This quotient space has some mild singularities and does not carry the structure of a smooth manifold but only of an infinite dimensional orbifold \cite{cervera1991action}.
However, for the purpose of this article we can ignore these subtleties and assume that we are always working away from the singularities, which allows us to treat $\mathcal{S}(M,\RR^3)$ as an infinite dimensional manifold.

By endowing the space of immersions $\Imm(M,\RR^3)$ with a Riemannian metric that is invariant under the actions of $\on{SO}(3)\ltimes\RR^3$ and $\Diff_+(M)$, the space of shapes $\mathcal{S}(M, \RR^3)$ becomes a Riemannian manifold (orbifold), where the metric is induced by the Riemannian metric on $\Imm(M, \RR^3)$.

In the following we will denote by $\on{dist}_{\Imm}$ the geodesic distance function of a Riemannian metric on the space of immersions $\Imm(M, \RR^3)$ and by $[f]$ the equivalence class of $f$ under the action of $\mathcal{G}$. Given two surfaces $f_1$ and $f_2$, we can define the distance between $[f_1]$ and $[f_2]$ as the infimum of the distance between the orbits of $f_1$ and $f_2$ under the action of $\mathcal{G}$. For example, the distance function on the space of unparametrized surfaces $\mathcal{S} = \Imm(M, \RR^3)/\Diff_+(M)$ can be defined as follows:
\begin{align}
\on{dist}_{\mathcal{S}}([f_1], [f_2]) = \inf_{\gamma\in\Diff_+(M)}\on{dist}_{\Imm}(f_1\circ\gamma, f_2).
\end{align}
We will use this induced  distance  as our measure for comparing unparametrized surfaces. Given two parametrized surfaces, to measure the similarity between them we will need to find the optimal reparametrization in $\Diff_+(M)$ that realizes the infimum. If we also want to mod out rigid motions and find the distance between two elements in the space of unparametrized surfaces modulo rigid motions $\Imm(M, \RR^3)/\left(\Diff_+(M)\times\on{SO}(3)\ltimes\RR^3\right)$, we will need to solve a joint optimization problem of finding the best reparametrization, rotation and translation.

\subsection{The General Elastic Metric and the SRNF Framework}\label{SRNF-map}
Jermyn et al. introduced in \cite{jermyn2012SRNF} the general elastic metric which has the desired invariance properties under shape-preserving deformations. To define this metric we first introduce a transformation that maps an immersion onto its induced surface metric and normal vector field:
\begin{align}
\on{Imm}(M,\mathbb R^3)&\mapsto \operatorname{Met}(M)\times C^{\infty}(M,\mathbb R^3) \\
f&\rightarrow \left(g:=g^f,n:=n^f\right)\;,
\end{align}
where $n^f$ is the unit normal vector field to the surface $f$, which is given in local coordinates by $$n=\frac{f_x\times f_y}{|f_x\times f_y|}$$ and where the surface metric is given by $$g=f^*\langle.,.\rangle_{\mathbb R^3}=\langle Tf.,Tf. \rangle_{\mathbb R^3}.$$
It is classical result in Riemannian geometry that any surface can be reconstructed uniquely by these two quantities \cite{kinetsu1975Gauss}. Thus, this representation allows one to define a Riemannian metric on the space of immersions by describing it on the image $\operatorname{Met}(M)\times C^{\infty}(M,\mathbb R^3)$. The general elastic metric as introduced in \cite{jermyn2012SRNF} is defined by:
\begin{multline}\label{eq.elasticMetric}
G_{g,n}((\delta g,\delta n),(\delta g, \delta n)) \\=  a \int_M \operatorname{tr} (g^{-1} \delta g g^{-1}\delta g) \mu_g +
b \int_M \operatorname{tr} (g^{-1} \delta g)^2  \mu_g\\ + c \int_M \langle \delta n, \delta n \rangle_{\mathbb R^3} \mu_g\;
\end{multline}
where $a,b,c\geq 0$ are constants and where $\mu_g$ denotes the induced volume density of the surface $f$.

Each of the three terms appearing in the metric \eqref{eq.elasticMetric} has a natural geometric interpretation: the first term penalizes local change in the metric (shearing),  the second term measures the change in the volume density (scaling) and the third term quantifies the change of the normal vector (bending).

Instead of using the $(g,n)$ representation for comparing surfaces, in the same paper \cite{jermyn2012SRNF} Jermyn et al. introduced the SRNF framework, where a surface is represented only as a rescaled normal vector field:
\begin{align*}
\mathcal{Q}: \on{Imm}(M, \RR^3)&\to C^{\infty}(M, \RR^3)\\
f(s)&\mapsto \sqrt{A(s)}n(s).
\end{align*}
where $A(s)$ denotes the local area-multiplication factor, which is given in local coordinates by $A(s)=|f_x(s)\times f_y(s)|$.
After equipping the target space $C^{\infty}(M, \RR^3)$ with the flat $L^2$ metric the map $Q$ becomes an infinitesimal isometry, where the space $\Imm(M, \RR^3)$ is equipped with the elastic metric $G^{a,b,c}$ with $a = 0, b = \frac{1}{16}$ and $c = 1$, i.e., the pullback of the $L^2$ metric on $C^{\infty}(M, \RR^3)$ along the map $Q$ is equal to the metric $G^{0,\frac{1}{16},1}$.
Note however that the resulting metric is degenerate for this choice of constants, i.e., there might exist deformation fields that have no cost with respect to the metric.  Furthermore, given $q\in C^{\infty}(M, \RR^3)$ there may be either no preimage $Q^{-1}(q) \in \Imm(M, \RR^3)$  of $q$ or many preimages. Most importantly the image of the space of immersions under the SRNF map cannot be easily characterized and, so far, it is not well understood.

Although the distance between two surfaces, which is given by the $L^2$ difference between their SRNFs, can be easily calculated, finding the inversion of the linear path between their SRNFs that realizes this distance is not possible as the linear path will usually leave the image of the SRNF-representation. In \cite{laga2017numerical} Laga et al. introduced a way to approximate the inversion of arbitrary paths between SRNFs by formulating inversion as an optimization problem.  In practice, this has been used to approximate geodesics, by numerically inverting straight lines between the SRNFs.
However, since the image of the SRNF-map is not convex in $L^2$ this method will not yield geodesics with respect to the SRNF metric, see Table~\ref{table.comparison}.
\subsection{Immersions and Vector Valued One-Forms}
In the following we will introduce our framework for comparing surfaces. The metric defined on the space of immersions can be seen as an alternative representation for the general elastic metric.
Therefore we consider the differential as a mapping
\begin{align}\label{differential}
d:\on{Imm}(M,\RR^3)/{\on{trans}}&\to\Omega_+^1(M,\RR^3)\\
f&\mapsto df\;,
\end{align}
where $\Omega_+^1(M,\RR^3)$ denotes the space of $\RR^3$-valued full-ranked one-forms on $M$. Given a metric $g$ on $M$, in a local chart with a field of orthonormal bases, an element of $\Omega_+^1(M,\RR^3)$ can be represented as a field of full-ranked $3\times 2$ matrices. The differential $d$ as defined above is  injective, but not surjective. Furthermore, in contrast to the SRNF mapping $\mathcal{Q}$ mentioned in Section~\ref{SRNF-map}, it is easy to characterize the image of the differential $d$. The following theorem contains this characterization and a result concerning the manifold structure space of full-ranked one forms $\Omega_+^1(M, \RR^3)$:
\begin{thm}
	The space of smooth full-ranked one-forms $\Omega_+^1(M, \RR^3)$ is an open subset of an infinite dimensional vector space $\Omega^1(M, \RR^3)$ and  thus it is an infinite dimensional Frechet manifold, where the tangent space at each point is simply $\Omega^1(M, \RR^3)$.
	
	Furthermore, the image of the differential $d$ is the space of all exact full-ranked one-forms, which is the intersection of $\Omega^1_
	+(M, \RR^3)$ with a linear subspace of $\Omega^1
	(M, \RR^3)$.
\end{thm}
\begin{proof}
The proof of this result follows directly from the definition of these spaces.
\end{proof}
This theorem allows us to define a Riemannian metric on these spaces as follows.  Let $\alpha\in\Omega_+^1(M, \RR^3)$ and $\xi\in T_{\alpha}\Omega_+^1(M, \RR^3)$. For the volume form $\mu$ on $M$ induced by the metric $g$ we let
\begin{align}\label{eq.metric}
G_{\alpha}(\xi,\xi) = \int_M\operatorname{tr}\left(\xi_x(\alpha_x^T\alpha_x)^{-1}\xi_x^T\right)\sqrt{\det(\alpha_x^T\alpha_x)}\mu.
\end{align}
This metric does not depend on the choice of orthonormal bases we choose and is actually independent of the metric $g$ on $M$, see \cite{bauer2018OneForms} for more details. Thus we can choose an easily obtainable metric $g$ on $M$ and then calculate this metric on $\Omega_+^1(M, \RR^3)$.

Using the injection \eqref{differential}, we  obtain a pullback metric on the space $\Imm(M, \RR^3)$ modulo translations and it turns out that this metric is related to the full elastic metric.
The space of immersions equipped with this inner product is an infinite dimensional Riemannian manifold. Unfortunately, there exists no explicit formula to calculate minimizing geodesics between two given immersions $f_0$ and $f_1$. Instead we will rely on numerical methods to minimize the path length over all paths of immersions connecting the given immersions $f_0$ and $f_1$. Alternatively these minimizing deformations can be found by solving the Lagrangian optimality condition for the energy functional, called the geodesic equation. Although we will not follow this strategy we will present this equation in Appendix~\ref{Appendix:geodesic_equation}.

First, however, we will orthogonally decompose the tangent space at $\alpha$ in a similar manner as in the definition of the elastic metric earlier.
Therefore we let
\begin{align*}
\xi =  \xi_m +  \frac12\tr(\alpha^+\xi)\alpha +\xi^{\perp} +\xi_0,
\end{align*}
where
\begin{align*}
\xi_m & =  \frac12\alpha(\alpha^T\alpha)^{-1}(\alpha^T\xi+\xi^T\alpha) - \frac12\tr(\alpha^+\xi)\alpha\\
\xi^{\perp} &= \xi - \alpha(\alpha^T\alpha)^{-1}\alpha^T\xi\\
\xi_0 &= \frac12\alpha(\alpha^T\alpha)^{-1}(\alpha^T\xi-\xi^T\alpha)
\end{align*}
(In the above formulae, we denote the Moore-Penrose inverse of $\alpha$ by $\alpha^+$. It is defined by $\alpha^+=(\alpha^T\alpha)^{-1}\alpha^T$ if $\alpha$ is a $3\times 2$ matrix of rank 2.)
It is easy to check that these terms are orthogonal with respect to the metric \eqref{eq.metric}. We can now obtain a family of metrics on $\Omega_+^1(M, \RR^3)$:
\begin{multline}\label{eq.metric_abcd}
G^{\mathfrak{a},\mathfrak{b},\mathfrak{c},\mathfrak{d}}_{\alpha}(\xi, \xi) \\= \mathfrak{a} G_{\alpha}(\xi_m, \xi_m) + \mathfrak{b} G_{\alpha}\left(\frac12\tr(\alpha^+\xi)\alpha, \frac12\tr(\alpha^+\xi){\alpha}\right)\\
\qquad + \mathfrak{c}G_{\alpha}( \xi^{\perp}, \xi^{\perp}) + \mathfrak{d}G_{\alpha}(\xi_0, \xi_0),
\end{multline}
where the first summand is measuring the deformation of the metric (within the class of metrics with the same volume form), the second summand is measuring the deformation of the volume density, the third summand is measuring the deformation of the normal vector and the last summand can be thought of as measuring the deformation of the local reparametrization.

The following theorem shows the connection of our split metric \eqref{eq.metric_abcd} with the elastic metric \eqref{eq.elasticMetric} on surfaces.
\begin{thm}\label{thm:correspondences}
	If $\mathfrak{d} = 0$, then the pull-back of the  split metric \eqref{eq.metric_abcd} gives rise to the elastic metric \eqref{eq.elasticMetric} on the space of immersions.
\end{thm}
\begin{proof}
	See Appendix~\ref{appendix.A} for a proof of this result.
\end{proof}

\begin{figure*}[!t]
	\centering
		\includegraphics[scale=0.45]{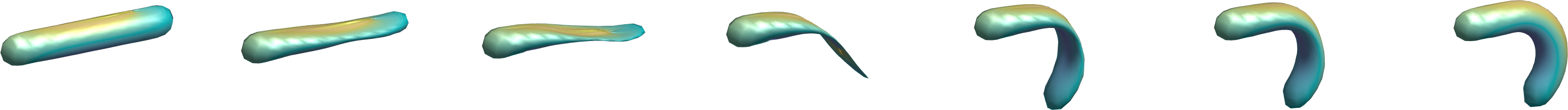}
		\vskip.1in
		\includegraphics[scale=0.45]{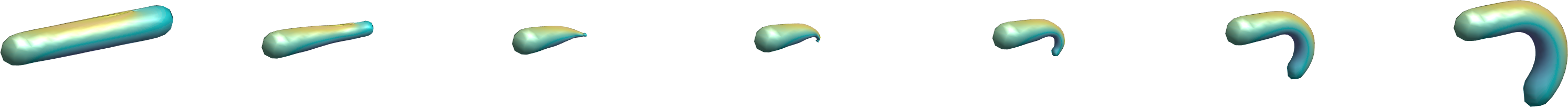}
		\vskip.1in
		\includegraphics[scale=0.45]{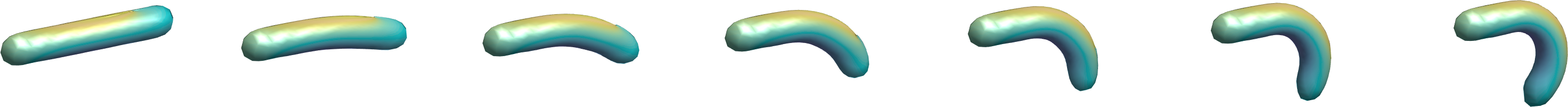}
		\vskip.1in
		\includegraphics[scale=0.45]{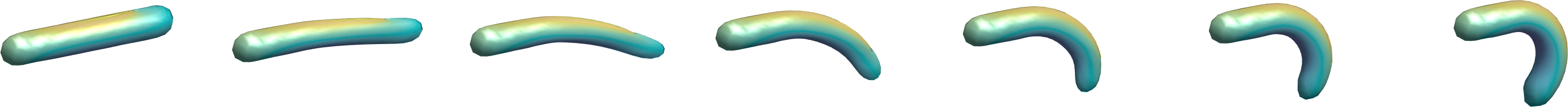}

	\caption{Geodesics between two cylinders in the space of immersions $\Imm(M, \RR^3)$ with respect to different choices of coefficients (from top to bottom): $ (1,1,0,1)$, $(1,0,1,1)$, $(1,1,1,0)$, $(0,\frac12,1,0)$.}
	\label{fig.geodesic.differentWeights2}
\end{figure*}

In Figure~\ref{fig.geodesic.differentWeights2}, we show geodesics between two parametrized cylinders with respect to the split metric \eqref{eq.metric_abcd} for different choices of coefficients $\mathfrak{a},\mathfrak{b},\mathfrak{c}$ and $\mathfrak{d}$. One can see how the choice of coefficients affects the resulting geodesic. Thus, in each specific application, we are now able to adjust the coefficients of the metric in a data driven way to obtain desired deformations between the shapes under consideration.

\begin{rem}
In  \cite{bauer2018OneForms} we have presented a detailed study of the metric~\eqref{eq.metric} on the space  $\Omega_+^1(M, \RR^3)$. In particular we have obtained an explicit formula for the corresponding geodesic initial value problem; in that situation geodesics can be computed pointwise, so the problem reduces to a finite-dimensional ODE which can be solved explicitly, and gives the solution in the infinite-dimensional context we are dealing with here.

	The space of exact one-forms $\Omega^1_{+,ex}(M,\RR^3)$ is, however,  a proper linear subspace of the space of non-singular one-forms $\Omega^1_+(S^2,\mathbb R^3)$, and is not a totally geodesic submanifold of $\Omega^1_+(S^2,\mathbb R^3)$ with respect to the metric~\eqref{eq.metric}.  As the space of immersions corresponds to the space of exact one-forms the obtained explicit formula for geodesics does not directly help to calculate geodesics on the space of immersions, which is the main goal of this article. In order to solve the geodesic problem we will thus introduce a discretization of the metric and solve the geodesic matching problem using path-straightening algorithms.
\end{rem}

Note that the split metric \eqref{eq.metric_abcd} is defined on differentials and thus is, by definition, independent of translations. To show the invariance of the split metric under rigid motions and diffeomorphisms, we now consider the action of the group of rotations $\on{SO}(3)$ on $\Omega_+^1(M,\RR^3)$, which is defined by pointwise left multiplication:
\begin{align}
\on{SO}(3)\times \Omega_+^1(M,\RR^3) &\to \Omega_+^1(M,\RR^3)\\
(R, \alpha)&\mapsto R\alpha,
\end{align}
where $(z\alpha)_x = R\alpha_x$; and the action of the group of diffeomorphisms $\Diff_+(M)$ on $\Omega_+^1(M,\RR^3)$, which is defined via pullback:
\begin{align}\label{eq.action.forms}
\Omega_+^1(M, \RR^3)\times\Diff_+(M) &\to \Omega_+^1(M, \RR^3)\\
(\alpha, \varphi)&\mapsto\varphi^*\alpha,
\end{align}
where
$
(\varphi^*\alpha)_x = \alpha_{\varphi(x)}\circ d\varphi_x.
$
The following proposition summarizes the most important invariances of the metric on $\Omega_+^1(M, \RR^3)$:
\begin{prop}\label{prop.invariances}
	Let $\alpha\in\Omega^1_+(M,\RR^3)$ and $\zeta, \eta\in T_\alpha\Omega^1_+(M,\RR^3)$.
	\begin{enumerate}
		\item The metric \eqref{eq.metric_abcd} is invariant under pointwise left multiplication with $\on{SO}(3)$. I.e., if $R \in \on{SO}(3)$, then
		\begin{equation*}
		G_{\alpha}(\zeta, \eta)=G_{R\alpha}(R\zeta,R\eta)
		\end{equation*}
		\item The  metric \eqref{eq.metric_abcd} is invariant under the right action of the diffeomorphism group, i.e., for any $\varphi \in \Diff_+(M)$ we have
		\begin{equation*}
		G_{\alpha}(\zeta, \eta)=G_{\varphi^*\alpha}(\varphi^*\zeta,\varphi^*\eta).
		\end{equation*}
	\end{enumerate}
\end{prop}
\begin{proof}
	The proof of the proposition follows exactly as for the metric \eqref{eq.metric}, which can be found in \cite{bauer2018OneForms}.
\end{proof}

\begin{figure*}[t!]
	\centering
		\includegraphics[scale=0.45]{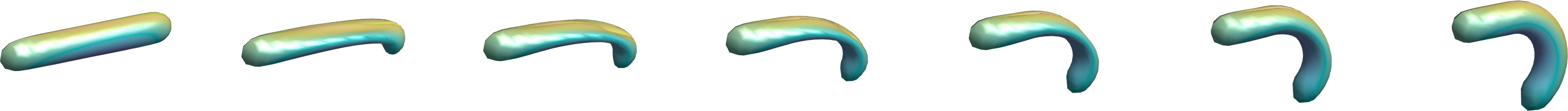}
		\vskip.1in
		\includegraphics[scale=0.45]{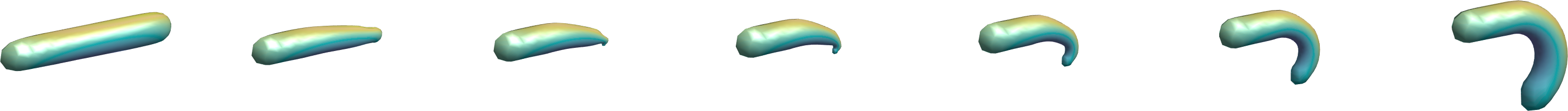}	
		\vskip.1in
		\includegraphics[scale=0.45]{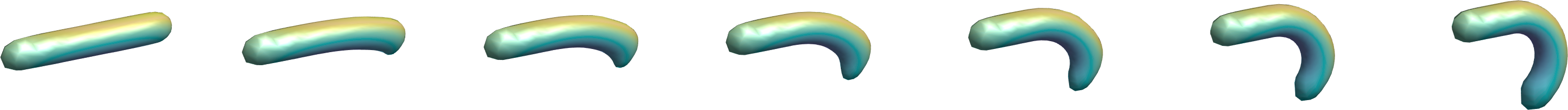}
		\vskip.1in
		\includegraphics[scale=0.45]{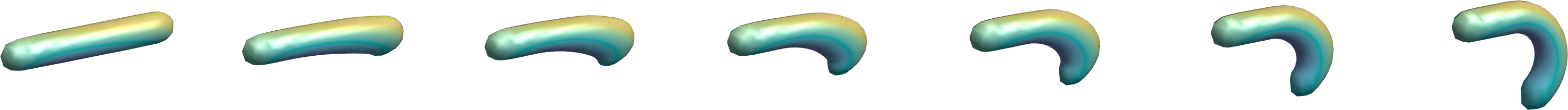}

	\caption{Geodesics between two cylinders in the space of unparametrized surfaces $\Imm(M,\RR^3)/\Diff_+(M)$ with respect to different choices of coefficients (from top to bottom): $ (1,1,0,1)$, $(1,0,1,1)$, $(1,1,1,0)$, $(0,\frac12,1,0)$.}
	\label{fig.geo.unpara.differentWeights2}
\end{figure*}

The group of rotations $\on{SO}(3)$ acts on the space of immersions by left multiplication, which is the same as it acts on the space of one forms.  Thus, by the first statement of Proposition~\ref{prop.invariances}, the pullback metric on $\Imm(M, \RR^3)$ is also invariant under the group of rigid motions $\on{SO}(3)\ltimes\RR^3$. For the standard action of $\Diff_+(M)$ by composition from the right on $\Imm(M, \RR^3)$, the following commutative diagram illustrates that the pull-back action of $\Diff_+(M)$ on $\Omega_+^1(M, \RR^3)$ is compatible with the action of $\Diff_+(M)$ on $\Imm(M, \RR^3)$:
\begin{align}
\xymatrix{ f  \ar[d]_{\text{action on}\ \Imm(M, \RR^3)}\ar[r]^{d}& df \ar[d]^{\text{action on}\ \Omega_+^1(M, \RR^3)} & \\
	f\circ\varphi  \ar[r]^{d} &**[r] \varphi^*df = df\circ d\varphi &}
\end{align}
Therefore, the second statement of Proposition~\ref{prop.invariances} gives the reparametrization invariance of the pullback metric on the space $\Imm(M,\RR^3)$.

Thus the metric on the space of immersions $\Imm(M, \RR^3)$ induces a metric on the space of unparametrized surfaces $\on{Imm}(M, \RR^3)/\Diff_+(M)$ and a metric on the space of unparametrized surfaces modulo rigid motions  $\on{Imm}(M, \RR^3)/\left(\Diff_+(M)\times\on{SO}(3)\ltimes\RR^3\right)$. In Figure~\ref{fig.geo.unpara.differentWeights2} we show geodesics between two cylinders in the space $\on{Imm}(M, \RR^3)/\Diff_+(M)$ with respect to the split metric \eqref{eq.metric_abcd} for different choices of coefficients $\mathfrak{a},\mathfrak{b},\mathfrak{c}$ and $\mathfrak{d}$. The corresponding geodesics in the space $\on{Imm}(M, \RR^3)/\left(\Diff_+(M)\times\on{SO}(3)\ltimes\RR^3\right)$ are shown in Figure~\ref{fig.geo.unparaModSO3.differentWeights2}.

\begin{figure*}[t!]
	\centering
		\includegraphics[scale=0.45]{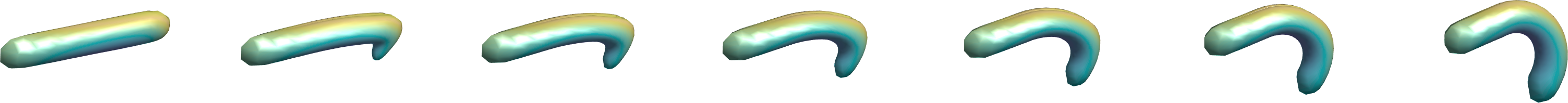}
		\vskip.1in
		\includegraphics[scale=0.45]{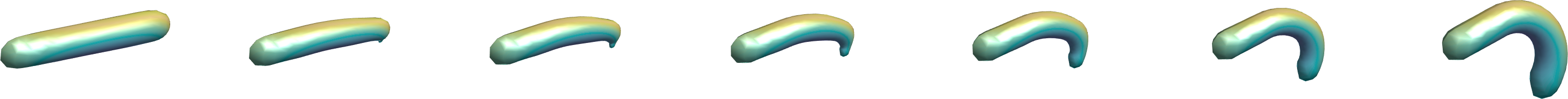}
		\vskip.1in
		\includegraphics[scale=0.45]{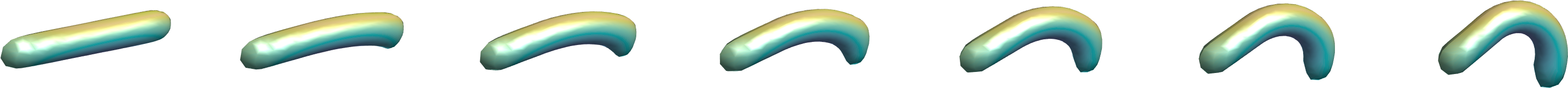}
		\vskip.1in
		\includegraphics[scale=0.45]{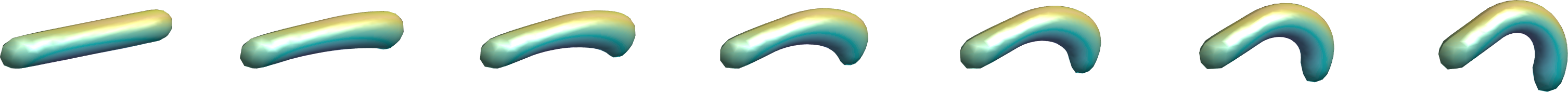}

	\caption{Geodesics between two cylinders in the space of unparametrized surfaces modulo rigid motions $\Imm(M,\RR^3)/\left(\Diff_+(M)\times\on{SO}(3)\ltimes\RR^3\right)$ with respect to different choices of coefficients (from top to bottom): $ (1,1,0,1)$, $(1,0,1,1)$, $(1,1,1,0)$, $(0,\frac12,1,0)$.}
	\label{fig.geo.unparaModSO3.differentWeights2}
\end{figure*}

\section{A numerical framework for the general elastic metric}
In this section we will describe the discretization and optimization procedure that we implemented to solve the geodesic boundary value problem.
From here on we assume that $M = S^2$ and  use a spherical coordinate system to represent an immersion $f: S^2\to\RR^3$ as a function $f:[0,2\pi]\times[0,\pi]\to \RR^3$ such that $f(0, \phi) = f(2\pi, \phi), f(\theta, 0) = f(0,0)$ and $f(\theta, \pi) = f(0, \pi)$, see Remark~\ref{rem:sphericalparametrization} below on how we obtain such (discrete) parametrizations in practice from a triangulated surface.

\begin{rem}\label{rem:sphericalparametrization}
We represent the surface of a given 3D shape with its embedding on a sphere $f:
S^2 \to \mathbb R^3$, which is always possible for genus-0 surfaces. In practice, methods such as conformal mapping introduce significant distortions when dealing with complex shapes that contain many elongated parts.  Since the proposed approach does not require the mapping to be conformal, we adopt the approach of Praun and Hoppe~\cite{praun2003spherical}, which has been implemented by Kurtek et.al.~\cite{kurtek2013landmark}.  The idea is to
progressively embed a surface on a sphere while
minimizing area distortion. The approach starts by reducing the mesh, using
progressive mesh simplification, to a basic polyhedra that can be
easily embedded on $S^2$. Then, it iteratively inserts vertices
and embeds each new vertex inside the spherical kernel of its
one-ring neighborhood while optimizing for the area distortion. The implementation provided in~\cite{kurtek2013landmark}, we reconstruct the mesh up to $1500$ vertices, which is sufficient for computing geodesics. This  procedure produces spherical maps that preserve important shape features as shown in all of the examples in this paper. Recall that, spherical parameterization of high genus surfaces is still an open problem.  Since we are
not aiming at solving the parameterization problem, we focus in this paper on
genus-0 manifold surfaces.
\end{rem}

The identity immersion $i: S^2\to\RR^3$ induces the spherical metric on $S^2$, which will serve as a background metric for the discretization; the vector fields $$\left\{\dfrac{1}{\sin\phi}\dfrac{\partial}{\partial\theta}, \dfrac{\partial}{\partial\phi}\right\}$$ form an orthonormal basis of the tangent space for any $(\theta, \phi)\in [0, 2\pi]\times(0, \pi)$.
With respect to this basis and the standard basis on $\RR^3$, the differential $df$ of an immersion $f = (x, y, z)^T$  can be represented by a field of $3\times 2$ matrices:
\begin{align*}
df\left(\dfrac{1}{\sin\phi}\dfrac{\partial}{\partial\theta}, \dfrac{\partial}{\partial\phi}\right) = \begin{pmatrix}
\dfrac{1}{\sin\phi}\dfrac{\partial x}{\partial\theta}, \dfrac{\partial x}{\partial\phi} \\
\dfrac{1}{\sin\phi}\dfrac{\partial y}{\partial\theta}, \dfrac{\partial y}{\partial\phi} \\
\dfrac{1}{\sin\phi}\dfrac{\partial z}{\partial\theta}, \dfrac{\partial z}{\partial\phi} \\
\end{pmatrix}.
\end{align*}
In the following we denote by $\norm{\cdot}_f$ the norm induced by the pullback of the split metric \eqref{eq.metric_abcd} and let $u\in T_f\Imm(S^2, \RR^3)$ be a tangent vector. Since $u$ can be seen as a function from $S^2$ to $\RR^3$, using this representation the norm of $u$ with respect to the split metric will be given as follows:
\begin{align}
\norm{u}_f = \big[G^{\mathfrak{a},\mathfrak{b},\mathfrak{c},\mathfrak{d}}_{df}(du, du)\big]^{1/2}.
\end{align}

\subsection{Geodesics in the space of surfaces}\label{subsec.geoImm}
We will now describe the solution of the boundary value problem in the pre-shape space of all parametrized surfaces.
Given two parametrized surfaces $f_1$ and $f_2$ we can discretize the linear path connecting $f_1$ and $f_2$ in $T$ time steps:
\begin{align}
f_{\text{lin}}(t_i) = (1 - t_i)f_1 + t_if_2.
\end{align}
where $t_i = i/T, i = 0,\ldots,T$. The differential $df_{\text{lin}}$ is then the linear path between $df_1$ and $df_2$, which stays by definition in the space of exact one-forms for all $i= 0,\ldots, T$. To solve the geodesic boundary value problem we will perturb $f(t)$ in all possible directions that fix the end points and that remain in the space of immersions. Since the map, as defined in equation~\eqref{differential}, is injective, this is equivalent to perturbing $df(t)$ in all possible directions in the space of exact one-forms that keep the  two boundary one-forms fixed.

To obtain a basis of perturbations in the space of immersions, we use the fact that the set of spherical harmonics in each component form a Hilbert basis of $L^2(S^2,\RR^3)$. We truncate this basis at a chosen maximal degree $\degree$ and denote the obtained set by $\{S_i\}$.
The number of elements in this basis is $L = 3((\degree+1)^2-1)$ (here we remove the spherical harmonic of degree 0  and order 0 since it is a constant function, which corresponds to a pure translation).
To calculate the optimal deformation between two given surfaces we aim to minimize the (discrete) path energy over all curves of the form
\begin{align}\label{eq:path}
f(t_0) &= f_1,\quad f(t_T) = f_2\\
f(t_i) &= (1 - t_i)f_1 + t_if_2 + \sum_{j =1}^L\on{Coeff}(j, i)S_j,
\end{align}
where $i = 1,\ldots,T-1$ and $\on{Coeff}$ is a $L\times (T-1)$ coefficient matrix.

The discrete energy functional $F: \mathbb R^{L\times (T-1)}\to \mathbb R$ is then given by
\begin{align}\label{eq.F.para}
F(\on{Coeff}) = \sum_{i =1}^T\norm{f_t(t_{i-1})}^2_{f(t_{i-1})}\Delta T
\end{align}
where the norm $\norm{\cdot}$ is induced by the pullback of the split metric \eqref{eq.metric_abcd},
\begin{align}\label{eq:derivative}
f_t(t_{i-1}) = \frac{f(t_{i}) - f(t_{i-1})}{\Delta T}
\end{align} is the (discrete) derivative of $f(t)$ at $f(t_{i-1})$ and $\Delta T = \frac{1}{T}$ is the width of a sub-interval. Alternatively one can also  discretize the derivative of $f$ using the central difference for interior data points, which makes the energy functional symmetric, but leads to slightly higher computational cost. To find the optimal coefficient matrix $\on{Coeff}$ we employ a BFGS method as provided in the optimize package of \emph{scipy}, where
we calculate the gradient using automatic differentiation in \emph{Pytorch}, which leads to the algorithm described in~Alg.\ref{algorithm:parametrized_surfaces}.

\begin{algorithm}
	\caption{The matching problem for
		parametrized surfaces}\label{algorithm:parametrized_surfaces}
	\begin{algorithmic}[1]		
		\renewcommand{\algorithmicrequire}{\textbf{Input:}}
		\renewcommand{\algorithmicensure}{\textbf{Output:}}
		\REQUIRE
		\hspace{10cm}
		\begin{enumerate}[1)]
			\item the source and target surfaces $f_1$ and $f_2$;
			\item the coefficients $(\mathfrak{a}, \mathfrak{b}, \mathfrak{c}, \mathfrak{d})$ of the metric;
			\item the number of time steps $T$;
			\item a basis $\{S_i, i=1,\ldots, L\}$ for the space of parametrized surfaces.
		\end{enumerate}
		\ENSURE
		\hspace{10cm}
		\begin{enumerate}[1)]
			\item the geodesic $f_{\operatorname{geo}}$ connecting $f_1$ to $f_2$;
			\item the geodesic distance $\operatorname{dist}$ between $f_1$, and $f_2$;
		\end{enumerate}
		\STATE Initialize $\on{Coeff}=0$ and $f(t_i)$ by equation~\eqref{eq:path}.
		\STATE Compute $f_t(t_{i-1})$ by equation~\eqref{eq:derivative}.
		\STATE Define the functional $F(\on{Coeff})$ as in equation~\eqref{eq.F.para}.
		\STATE Minimize $F$ using a BGFS-method, where the gradient of $F$ with respect to $\on{Coeff}$ is caluclated using the automatic differentiation package in Pytorch.
		\STATE Set \begin{align}
		f_{\on{geo}}(t_0) &=  f_1, \qquad f_{\on{geo}}(t_T) = f_2\\
		f_{\on{geo}}(t_i) &= (1-t_i)f_1+t_if_2 +\sum_{j =1}^L\on{Coeff}(j, i)S_j
		\end{align}
		and
		$\on{dist}=\sqrt{F(\on{Coeff})}$.
		\RETURN $f_{\on{geo}}$ and $\on{dist}$.	
	\end{algorithmic}
\end{algorithm}

\subsection{Geodesics in the space of unparametrized surfaces}\label{subsec.geoUnpara}
Now we present our algorithm for calculating geodesics in the space of unparametrized surfaces $\Imm(S^2, \RR^3)/\Diff_+(S^2)$. The main difficulty for this task is to find the optimal $\gamma\in\Diff_+(S^2)$ that realizes the distance
\begin{align}\label{dist:shape_space}
\on{dist}_{\mathcal{S}}([f_1], [f_2]) = \inf_{\gamma\in\Diff_+(S^2)}\on{dist}_{\Imm}(f_1\circ\gamma, f_2),
\end{align}
where $[f]$ is the equivalence class of $f$ under the action of the group of orientation-preserving diffeomorphisms $\Diff_+(S^2)$ and
$\on{dist}_{\mathcal{S}}$ denotes the distance function on the space $\Imm(S^2, \RR^3)/\Diff_+(S^2)$ with respect to the metric that is induced from the split metric \eqref{eq.metric_abcd}.

In order to practically perform the minimization over the infinite dimensional space $\Diff_+(S^2)$ we have to choose a suitable discretization of this group:
Let $\on{Id}$ be the identity map from $S^2$ to itself.  The tangent space $T_{\on{Id}}\Diff_+(S^2)$ is the set of all (smooth) vector fields on $S^2$.
It is known that the set of gradient and skew gradient vector fields of the set of spherical harmonics provide an orthogonal basis for this tangent space -- here orthogonal means with respect to the standard $L^2$ metric. Normalizing these basis we obtain an orthonormal basis for the tangent space $T_{\on{Id}}\Diff_+(S^2)$. To choose a finite dimensional discretization of the tangent space, we truncate this basis at a maximal degree $\degbar$, then the number of elements in this basis is $\bar{L} = 2(\degbar+1)^2-2$. From here on we will denote this truncated basis  by $\{v_i, i = 1,...,\bar{L}\}$.
Let $X^v = (X^v_1, X^v_2,\ldots,X^v_{\bar{L}})$ be the coefficients of a vector field with respect to this basis and consider the induced mapping
\begin{align}\label{eq.gamma}
\gamma = \on{Proj}\left(\on{Id} + \sum_{k=1}^{\bar{L}}X^v_kv_k\right),
\end{align}
where $\on{Proj}$ denotes the map that projects non-zero vectors in $\mathbb R^3$ onto the unit sphere $S^2\subset \mathbb R^3$.
The following result gives an explicit bound on the size of $X^v$, that ensures that the corresponding $\gamma$, defined by~\eqref{eq.gamma}, is a diffeomorphisms of $S^2$.

\begin{thm}\label{Proj_diff}
	Let $ U= \sum_{k=1}^{\bar{L}}X^v_kv_k$ be a vector field on the sphere $S^2$.
	and let $\gamma= \on{Proj}\left(\on{Id} + t U \right)$ be the corresponding map as defined in \eqref{eq.gamma}, for some real $t$.
	Then $\gamma$ is a diffeomorphism if
	\begin{align}\label{tcondition}
	\lvert t\rvert  <  -\frac{1}{\inf_{p\in M} \lambda_-(\nabla U)},
	\end{align}
	where $\nabla U$ is the $(1,1)$ tensor field $v\mapsto \nabla_vU$ and $\lambda_-(\nabla U)$ is the smaller of the two real eigenvalues of the symmetrized matrix $\overline{\nabla U} = \tfrac{1}{2} \big(\nabla U + (\nabla U)^T\big)$. 
\end{thm}
\begin{proof}
	The proof of this result is postponed to Appendix \ref{appendix.A}. Note that since $\mathrm{Tr}{(\nabla U)} = \mathrm{div}{U}$, which integrates to zero over the compact manifold $M$, we know that $\lambda_-(\nabla U)$ is always negative somewhere; hence the bound on $\lvert t\rvert$ is some positive number.
\end{proof}

We are now able to describe the discrete optimization problem on the space of unparametrized surfaces, i.e., we aim to minimize the discrete functional
$\bar F: \RR^{\bar L+ L\times(T-2)}\to \mathbb R$ given by
\begin{align}\label{eq.barF.unpara}
\bar F(X^v, \on{Coeff}) = \sum_{i =1}^T\norm{f_t(t_{i-1})}^2_{f(t_{i-1})}\Delta T,
\end{align}
where the norm $\norm{\cdot}$ is induced by the pullback of the split metric \eqref{eq.metric_abcd}, $\on{Coeff}$, $S_i$, $\Delta T = \frac{1}{T}$ are as in Subsection~\ref{subsec.geoImm} and where the discrete curve $f$ is now of the form
\begin{align}\label{eq:unpara.path}
f(t_0) &= f_1\circ\gamma,\quad f(t_T) = f_2\\
f(t_i) &= (1 - t_i)f_1 + t_if_2 + \sum_{j =1}^L\on{Coeff}(j, i)S_j,
\end{align}
and where the reparametrization $\gamma$ is given by formula \eqref{eq.gamma} with  coefficient vector  $X^v = (X^v_1, X^v_2,\ldots,X^v_{\bar L})$.

\begin{rem}[Initialization over $\Diff_+(S^2)$]\label{rem:initialize}
	When using a gradient based optimization method, it is always an important issue to find a good initialization, as the optimization procedure can get stuck in local minima and is usually sensitive to this initialization.
	In order to find a good initial guess for the optimal reparametrization of the surface $f_1$, we first align the corresponding SRNFs of the two boundary surfaces $f_1$ and $f_2$. This seems a natural initialization for the $(0,\frac12,1,0)$ metric as the $L^2$-distance on the space of SRNFs is a first order approximation of the geodesic distance
	of this metrics. However, in all our experiments it turned out that this initialization works well for other choices of constants as well, as the optimal point correspondences for different choices of constants, albeit different,
	are still similar on a global scale. Furthermore,  we note that any three dimensional rotation can be seen as a diffeomorphism of $S^2$. We use this fact to first minimize
	only over this finite dimensional subgroup of the infinite dimensional reparametrization group. Finally, to initialize the optimization over this finite dimensional group, we first consider
	the icosahedral group, which contains $60$ orientation preserving rotations denoted by $h_i, i = 1,\ldots, 60$, as a finite subset of $\on{SO}(3)$. We then choose the best diffeomorphism among these 60 elements as our initial guess.
	
\end{rem}

\begin{figure}[ht]
	\centering
		\includegraphics[scale=0.55]{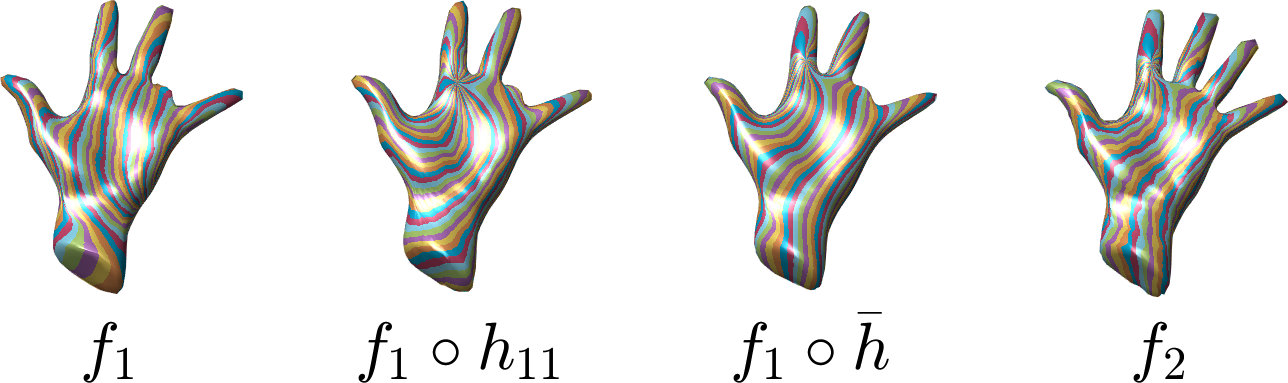}
		\vskip.2in
		\includegraphics[scale=0.55]{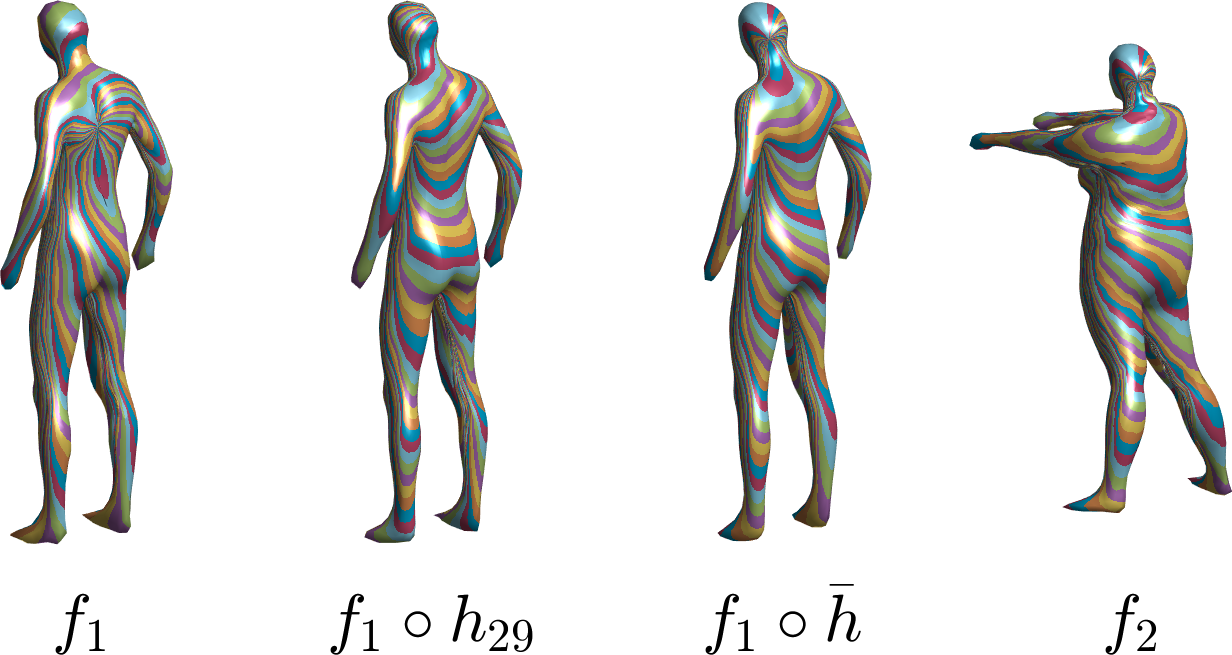}

	\caption{Examples of boundary surfaces before and after the optimization over the reparametrization group with respect to the split $(1,1,1,0)$ metric. Here the second shape shows the parametrization of the first boundary surface after composing by the initial guess in the icosahedral group and the third shape shows the final point correspondences after the full optimization, where $\bar h$ denotes the optimal reparametrization. One can observe how the parametrization of the initial surface successively better matches the parametrization of the target surface (the color map represents the parametrization of the surfaces).}
	\label{fig.initializationOverDiffS2}
\end{figure}

In the following we will describe two algorithms for calculating geodesics in the space of unparametrized surfaces $\Imm(S^2, \RR^3)$: a \emph{joint optimization} procedure and a \emph{coordinate descent} approach, where we minimize alternating in the space of parametrized surfaces and over the reparametrization group separately.

We will start by describing the joint optimization procedure, which is analogous to the optimization for parametrized surfaces with one caveat:
since  formula \eqref{eq.gamma} only leads to diffeomorphisms near the identity, i.e., reparametrizations that map  points on $S^2$ to nearby points,  we will  describe large deformations between $S^2$ as a composition of $N$ such (small) deformations. This will lead us to iteratively solve the joint optimization problem. The corresponding algorithm is described in Alg.\ref{alg.jointopt}.
\begin{algorithm}
	\caption{The joint optimization approach}
	\begin{algorithmic}[1]\label{alg.jointopt}
		\renewcommand{\algorithmicrequire}{\textbf{Input:}}
		\renewcommand{\algorithmicensure}{\textbf{Output:}}
		\REQUIRE
		\hspace{10cm}
		\begin{enumerate}[1)]
			\item the source and target surfaces $f_1$ and $f_2$;
			\item the coefficients $(\mathfrak{a}, \mathfrak{b}, \mathfrak{c}, \mathfrak{d})$ of the metric;
			\item the number of time steps $T$;
			\item bases $\{S_i,i=1,\ldots,L\}$ and $\{v_i,i=1,\ldots,\bar L\}$ for the space of parametrized surfaces and vector fields on $S^2$ resp.;	
			\item the number $N$ that describes the maximal amount of small deformations used.		
		\end{enumerate}
		\ENSURE 		\hspace{10cm}
		\begin{enumerate}[1)]
			\item the geodesic $f_{\operatorname{geo}}$ connecting $[f_1]$ to $[f_2]$;
			\item the geodesic distance $\operatorname{dist}$ between $[f_1]$ and $[f_2]$;
		\end{enumerate}
		\STATE Initialize  $\bar f = f_1$, $\on{Coeff} = 0$
		\WHILE{$k \leq N $}
		\STATE Initialize $\gamma$ by formula \eqref{eq.gamma} with $X^v=0$ and $f(t_i)$ by equation \eqref{eq:unpara.path}.
		\STATE Compute $f_t(t_{i-1})$ by equation \eqref{eq:derivative}.
		\STATE Define the functional $\bar F(X^v, \on{Coeff})$ by \eqref{eq.barF.unpara}
		where the discrete curve $f$ is of the form
		\begin{align}
		f(t_0) &= \bar f\circ\gamma,\quad f(t_T) = f_2\\
		f(t_i) &= (1 - t_i)f_1 + t_if_2 + \sum_{j =1}^L\on{Coeff}(j, i)S_j,
		\end{align}
		\STATE Minimize $\bar F$ using a BFGS method, where the gradients of $\bar F$ with respect to $X^v$ and $\on{Coeff}$ are calculated using the automatic differentiation package in Pytorch.
		\STATE Compute the optimal $\gamma$ using formula \eqref{eq.gamma}.
		\STATE Update $\bar f = \bar f\circ\gamma$.
		\STATE $k = k+1$
		\ENDWHILE
		\STATE Set \begin{align}
		f_{\on{geo}}(t_0) &= \bar f, \qquad f_{\on{geo}}(t_T) = f_2\\
		f_{\on{geo}}(t_i) &= (1-t_i)f_1+t_if_2 +\sum_{j =1}^L\on{Coeff}(j, i)S_j,
		\end{align}
		and $\on{dist} = \sqrt{\bar F(X^v, \on{Coeff})}$.
		\RETURN $f_{\on{geo}}$ and $\on{dist}$
	\end{algorithmic}
\end{algorithm}

As an alternative to the joint optimization we will present in the following a \emph{coordinate descent method}, where we separate the variables in the space of surfaces from the variables that govern the reparametrization of the initial surface, i.e., we alternate between calculating a discrete geodesic, denoted by $f_{\on{opt}}$, between the parametrized surfaces $f_1$ and $f_2$ in the space of immersions $\Imm(S^2, \RR^3)$ and reparametrizing the initial surface $\bar f=f_1$. To  update the reparametrization  we consider only the first two time points of $f_{\on{opt}}$, i.e., $\bar f$ and $f_{\on{opt}}(t_1)$, and define the following functional
\begin{align}\label{eq.Fr}
F_r(X^v) = \norm{f_{\on{opt}}(t_1) - \bar f\circ\gamma}_{\bar{f}\circ\gamma}^2,
\end{align}
where $\gamma$ is given by formula \eqref{eq.gamma} and $X^v=(X^v_1, X^v_2,\ldots,X^v_{\bar L})$.
We can now employ a BFGS method to find the optimal coefficient vector $X^v_{\on{opt}}$, compute $\gamma$ using formula \eqref{eq.gamma}, and then update $\bar f = \bar f\circ\gamma$. Then we repeat this process by recalculating the geodesic in the space of parametrized surfaces (with the changed initial surface $\bar f$). The whole optimization process is summarized in Alg. \ref{alg.coordesc}.

\begin{algorithm}
	\caption{The coordinate descent approach}
	\begin{algorithmic}[1]\label{alg.coordesc}
		\renewcommand{\algorithmicrequire}{\textbf{Input:}}
		\renewcommand{\algorithmicensure}{\textbf{Output:}}
		\REQUIRE
		\hspace{10cm}
		\begin{enumerate}[1)]
			\item the source and target surfaces $f_1$ and $f_2$;
			\item the coefficients $(\mathfrak{a}, \mathfrak{b}, \mathfrak{c}, \mathfrak{d})$ of the metric;
			\item the number of time steps $T$;
			\item bases $\{S_i,i=1,\ldots,L\}$ and $\{v_i,i=1,\ldots,\bar L\}$ for the space of parametrized surfaces and vector fields on $S^2$ resp;	
			\item the number $N$ that describes the maximal amount of small deformations used.		
		\end{enumerate}
		\ENSURE 		\hspace{10cm}
		\begin{enumerate}[1)]
			\item the geodesic $f_{\operatorname{geo}}$ connecting $[f_1]$ to $[f_2]$;
			\item the geodesic distance $\operatorname{dist}$ between $[f_1]$ and $[f_2]$;
		\end{enumerate}
		\STATE Let $\bar f = f_1$ and initialize $\on{Coeff}=0$.
		\STATE Choose a positive integer $N$.
		\WHILE{$k \leq N $}
		\STATE Define the functional $F(\on{Coeff})$ by~\eqref{eq.F.para}
		where the discrete curve $f$ is of the form
		\begin{align}
		f(t_0) &= \bar f,\quad f(t_T) = f_2\\
		f(t_i) &= (1 - t_i)f_1 + t_if_2 + \sum_{j =1}^L\on{Coeff}(j, i)S_j,
		\end{align}
		\STATE Minimize $F\big(\on{Coeff})$ using a BFGS method, where the gradient of $F$ with respect to $\on{Coeff}$ is calculated using the automatic differentiation package in Pytorch.
		\STATE Calculate $f_{\on{opt}}(t_1) = \sum_{i=1}^{L}\on{Coeff}(i,1)S_i$.
		\STATE Initialize $X^v = 0$ and $\gamma$ by formula \eqref{eq.gamma}.
		\STATE Define the functional $F_r(X^v)$ by equation \eqref{eq.Fr}.
		\STATE Minimize $F_r$ using a BFGS method with gradient of $F_r$ with respect to $X^v$ calculated using the automatic differentiation package.
		\STATE Compute $\gamma$ using formula \eqref{eq.gamma}.
		\STATE Update $\bar f = \bar f\circ\gamma$.
		\STATE $k = k+1$
		\ENDWHILE
		\STATE Set \begin{align}
		f_{\on{geo}}(t_0) &= \bar f, \qquad f_{\on{geo}}(t_T) = f_2\\
		f_{\on{geo}}(t_i) &= (1-t_i)f_1+t_if_2 +\sum_{j =1}^L\on{Coeff}(j, i)S_j,
		\end{align}
		and $\on{dist} = \sqrt{F(\on{Coeff})}$.
		\RETURN $f_{\on{geo}}$ and $\on{dist}$
	\end{algorithmic}
\end{algorithm}

\subsection{Geodesics in the space of unparametrized surfaces modulo rigid motions}\label{subsec.geoUnparaModSO3}
Note that the split metric \eqref{eq.metric_abcd} associates no cost to translation and thus the obtained geodesic is automatically in the space of surfaces modulo translations. To calculate the geodesic between two surfaces $[f_1]$ and $[f_2]$ in the space of unparametrized surfaces modulo rigid motions $\Imm(S^2, \RR^3)/\left(\Diff_+(S^2)\times\on{SO}(3)\ltimes\RR^3\right)$, we will need to minimize in addition over the rotation group, i.e.,
solve the optimization problem on $\on{SO}(3)\times\Diff_+(S^2)$:
\begin{align}
\on{dist}_{\mathcal{S}}([f_1], [f_2])
&= \inf_{\substack{R\in \on{SO}(3) \\ \gamma\in\Diff_+(S^2)}}\on{dist}_{\on{Imm}}(f_1\circ\gamma, Rf_2),
\end{align}
where $[f]$ is the equivalence class of $f$ under the actions of $\Diff_+(S^2)$ and $\on{SO}(3)$ and $\on{dist}_{\mathcal{S}}$  denotes the distance function on the space $\Imm(S^2, \RR^3)/\left(\Diff_+(S^2)\times\on{SO}(3)\ltimes\RR^3\right)$.

Let $\norm{\cdot}, \on{Coeff}, S_i, \Delta T$ be as in Subsection~\ref{subsec.geoImm} and let $\bar f $ be the current parametrization of the first boundary surface. It is known that the group of rotations $\on{SO}(3)$ is a three dimensional Lie group and the matrix exponential from its Lie algebra $\mathfrak{so}(3)$ is surjective. Since there is an isomorphism between $\RR^3$ and $\mathfrak{so}(3)$, the discrete optimization problem on the space of unparametrized surfaces modulo rigid motions will be minimizing the discrete functional $\tilde F:\RR^{3+ \bar L+ L\times(T-2)}\to \RR$ given by
\begin{align}
\tilde F(R, X^v, \on{Coeff}) = \sum_{i =1}^{T}\norm{f_t(t_{i-1})}^2_{f_{i-1}}\Delta T,
\end{align}
where the discrete curve in this case is of the form
\begin{align}\label{eq:unpara.modRig.path}
f(t_0) &= \bar f\circ\gamma,\quad f(t_T) = Rf_2\\
f(t_i) &= (1 - t_i)f_1 + t_if(t_T) + \sum_{j =1}^L\on{Coeff}(j, i)S_j,
\end{align}
$i = 1,\cdots, T-1$ and where the reparametrization $\gamma$ is given by formula \eqref{eq.gamma} with  coefficient vector  $X^v = (X^v_1, X^v_2,\ldots,X^v_{\bar L})$. We will tackle this simpler (finite dimensional) optimization problem using an analogous approach as in the previous section and will thus omit further details.

\section{Experiments}
\begin{figure*}[htbp]
	\centering
	\includegraphics[scale=0.45]{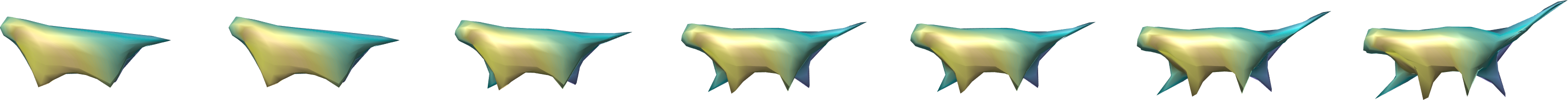}\vskip.25in
	\includegraphics[scale=0.45]{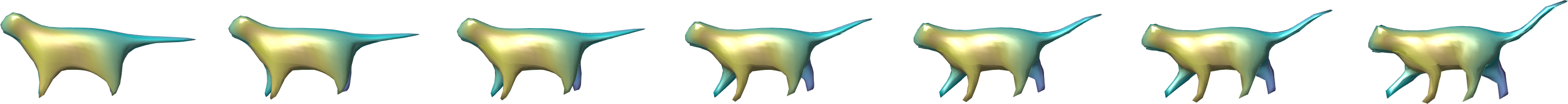}
	\vskip.25in
	\includegraphics[scale=0.45]{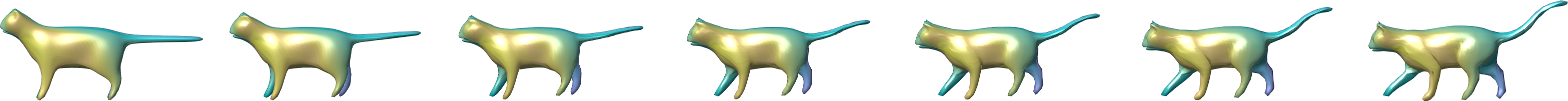} \vskip.25in
	\caption{Example of a geodesic in several resolutions: $12\times 25$ (top), $25\times 49$ (middle) and $50\times 99$ (bottom) with respect to the split $(1,1,0.1,0)$ metric in the space of unparametrized surfaces $\Imm(S^2,\RR^3)/\Diff_+(S^2)$, where $\degree = 7, \degbar= 7$ and $T = 13$.}
	\label{fig.geodesic.multires}
\end{figure*}
 In this section we will present examples of geodesics as calculated using our optimization procedures. 
The human body shapes have been kindly provided by Nil Hasler~\cite{hasler2009statistical} and the hand shape is taken from SHREC07 watertight models.
All other shapes are are courtesy of the TOSCA shape data base \cite{bronstein2008numerical}.

\subsection{Geodesics and Karcher Mean}
In Figure~\ref{fig.geodesic.energy.shape} we present examples of geodesics between given surfaces in the space $\Imm(S^2, \RR^3)/\Diff_+(S^2)$ with respect to the split $(1,1,0.1,0)$ metric and the corresponding evolutions of energies. In all our examples we observed a good and relatively fast convergence of the optimization procedure, and we present some selected results of the resulting deformation and the corresponding computation times in Table~\ref{table.time.multiresResults}. In Figure~\ref{fig.KarcherMean} we present 
  the Karcher mean of a family of cat surfaces with respect to the split $(1,1,0.1,0)$ metric in the space of unprametrized surfaces modulo rigid motions $\Imm(S^2, \RR^3)/(\Diff_+(S^2)\times\on{SO(3)\ltimes\RR^3})$. One can observe that the mean captures the overall characteristics of the family of surfaces under consideration, but simplifies some of the features that undergo high variabilty. 
All results were obtained on a standard laptop without any parallelization or GPU-implementation, which could certainly be used to obtain a significant increase in speed.

\begin{figure}[ht]
	\centering
	\includegraphics[scale=0.30]{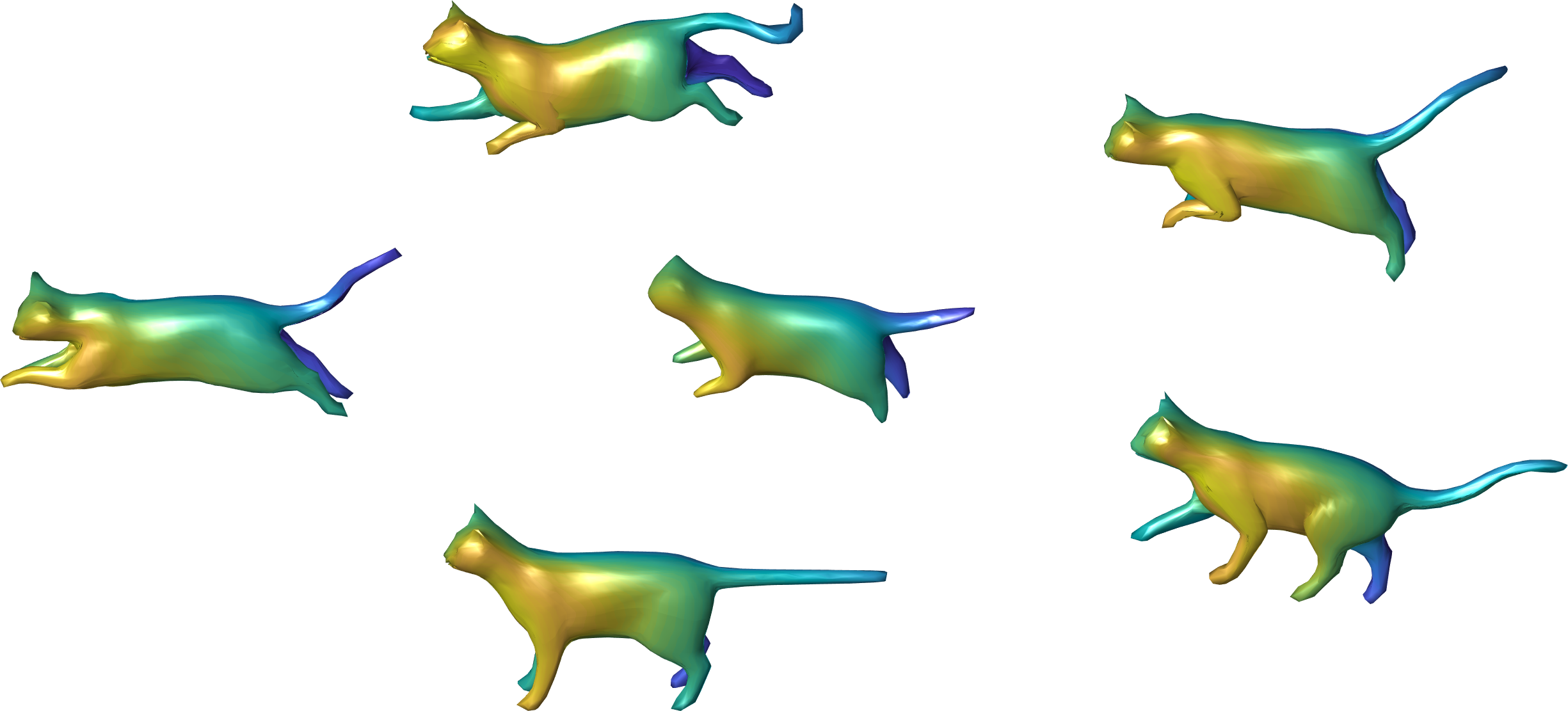}
	\caption{The Karcher mean (middle) of a set of shapes of cats in the space $\Imm(S^2, \RR^3)/(\Diff_+(S^2)\times\on{SO(3)\ltimes\RR^3})$ with respect to the split $(1,1,0.1,0)$ metric.}
	\label{fig.KarcherMean}
\end{figure}

\begin{table} [ht]
	\begin{adjustbox}{width=\columnwidth,center}
		\begin{tabular}{ | c | c | c | c | }
			\hline
			$\raisebox{.15cm}{\phantom{M}}\raisebox{-.15cm}{\phantom{M}}$
			Boundary Surfaces & Resolution & Iter & RunTime \\ \hline
			{\multirow{3}{*}{\vspace{-.3cm}\hspace{0cm}\includegraphics[scale=.13]{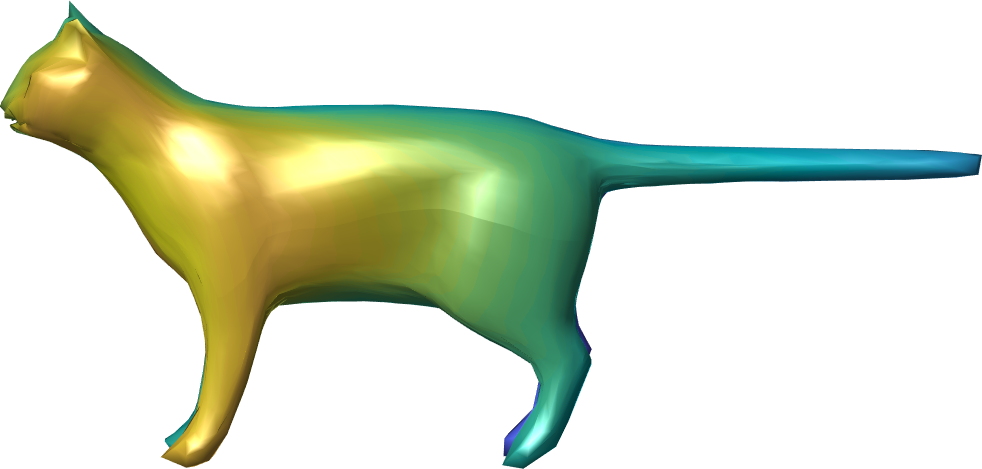}\hskip.1in\includegraphics[scale=.13]{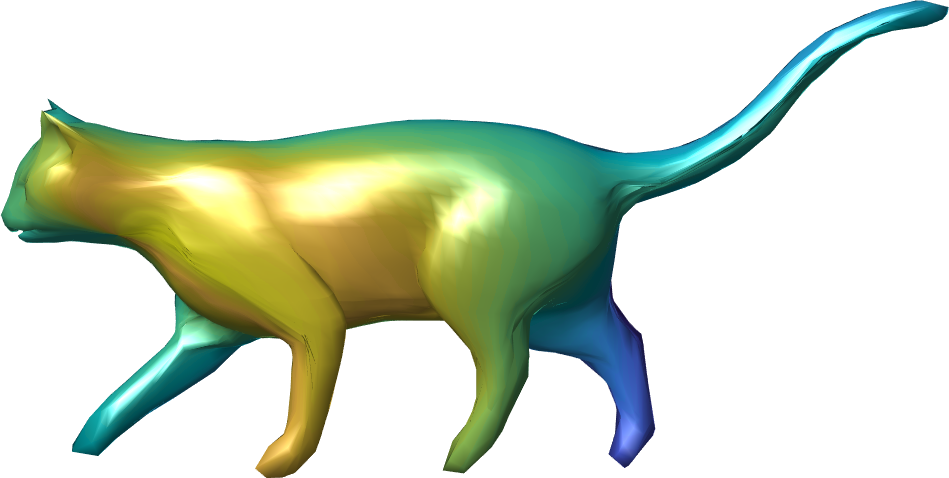}}}
			& $\raisebox{.15cm}{\phantom{M}}\raisebox{-.15cm}{\phantom{M}}$
			low & 114 & 39.7s \\ \cline{2-4}
			& $\raisebox{.15cm}{\phantom{M}}\raisebox{-.15cm}{\phantom{M}}$
			middle & 237 & 3min 3s \\ \cline{2-4}
			& $\raisebox{.15cm}{\phantom{M}}\raisebox{-.15cm}{\phantom{M}}$
			high & 235 & 14min 2s \\ \hline
			{\multirow{3}{*}{\vspace{-.6cm}\hspace{0cm}\includegraphics[scale=.15]{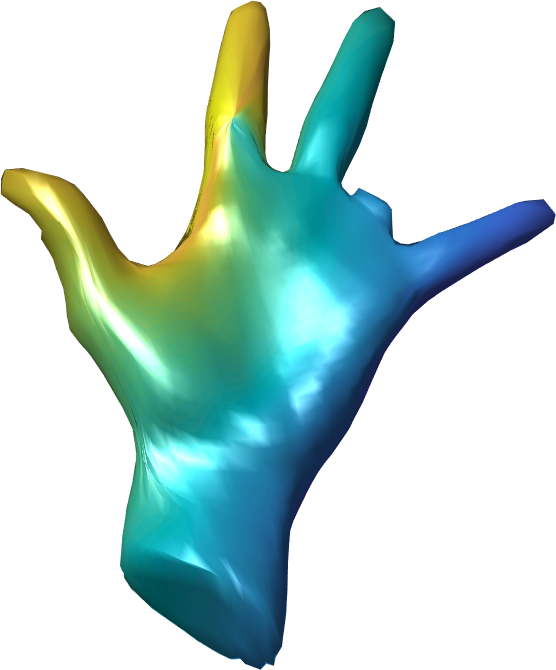}\hskip.2in\includegraphics[scale=.15]{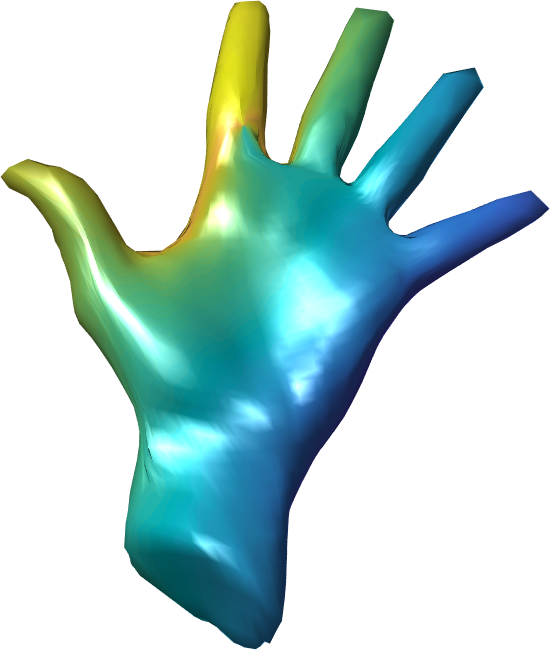}}}
			& $\raisebox{.25cm}{\phantom{M}}\raisebox{-.15cm}{\phantom{M}}$
			low & 42 & 40.7s\\ \cline{2-4}
			& $\raisebox{.25cm}{\phantom{M}}\raisebox{-.15cm}{\phantom{M}}$
			middle & 113 & 1min 35s\\ \cline{2-4}
			& $\raisebox{.25cm}{\phantom{M}}\raisebox{-.25cm}{\phantom{M}}$
			high & 139 & 8min 25s\\ \hline
			{\multirow{3}{*}{\vspace{-.4cm}\hspace{0cm}\includegraphics[scale=.17]{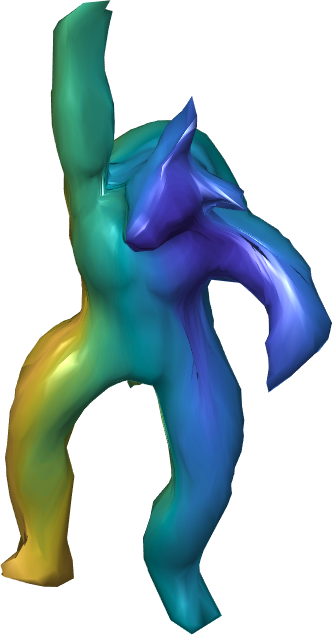}\hskip.3in\includegraphics[scale=.17]{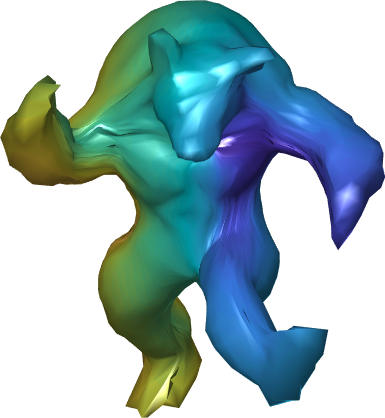}}}
			& $\raisebox{.15cm}{\phantom{M}}\raisebox{-.15cm}{\phantom{M}}$
			low & 88 & 32.5s\\ \cline{2-4}
			& $\raisebox{.15cm}{\phantom{M}}\raisebox{-.15cm}{\phantom{M}}$
			middle & 220 & 2min 22s\\ \cline{2-4}
			& $\raisebox{.15cm}{\phantom{M}}\raisebox{-.15cm}{\phantom{M}}$
			high & 193 & 10min 27s\\ \hline
		\end{tabular}
	\end{adjustbox}
	\caption{Numerical results of matching surfaces with different resolutions in time and space: low: $12 \times 25$, $\degree = \degbar = 5$, $T = 5$; middle: $25\times 49$, $\degree = 7$, $\degbar = 8$, $T = 10$; high: $50 \times 99$, $\degree = 9$, $\degbar = 11$, $T = 15$. Here Iter denotes the number of iterations until convergence in the optimization process.}
	\label{table.time.multiresResults}
\end{table}

\begin{figure*}[ht]
	\centering
	\includegraphics[scale=0.45]{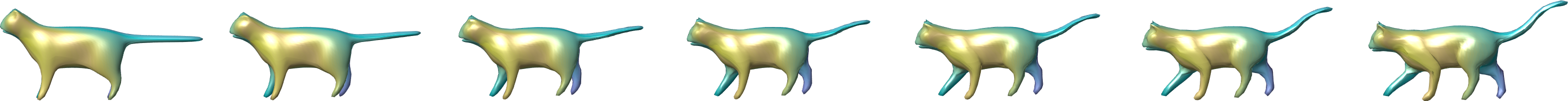}\vskip.25in
	\includegraphics[scale=0.45]{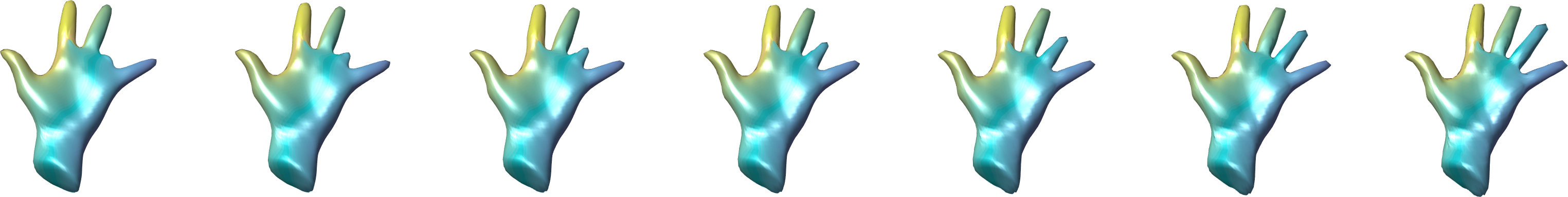}
	\vskip.145in
	\includegraphics[scale=0.45]{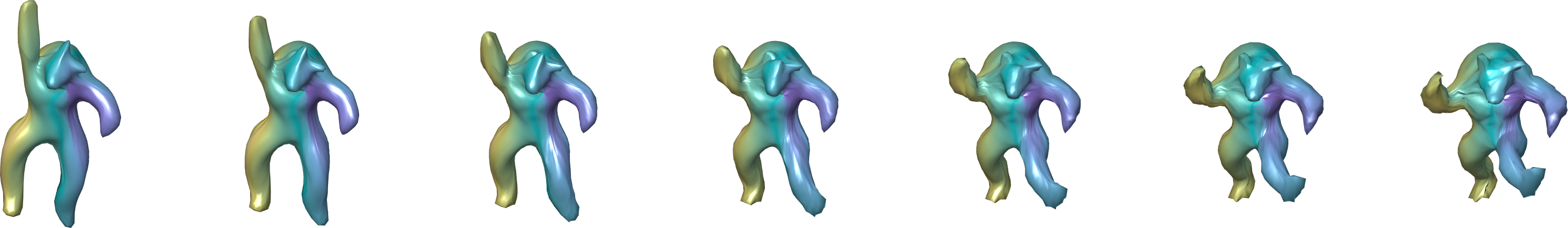} \vskip.25in
	
	\includegraphics[scale=0.35]{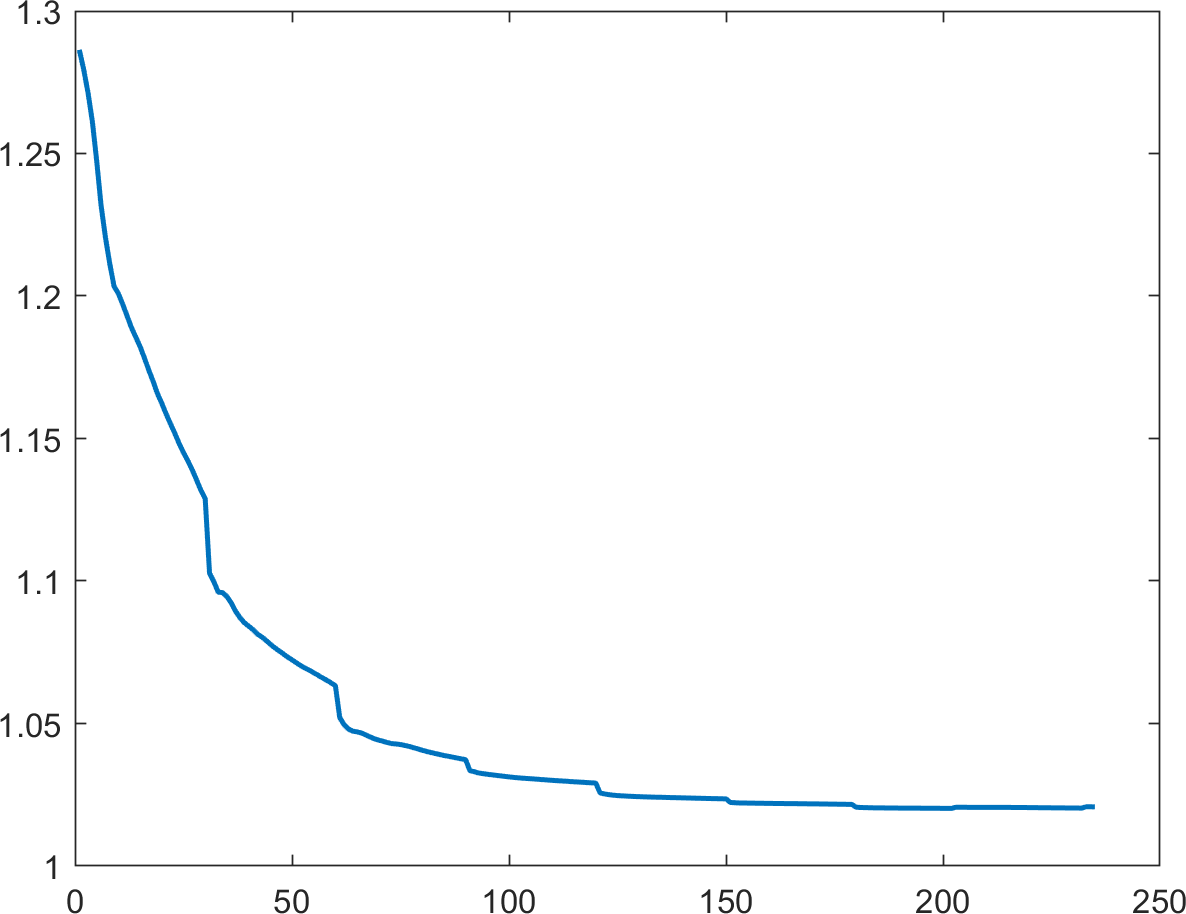}\hskip.35in
	\includegraphics[scale=0.35]{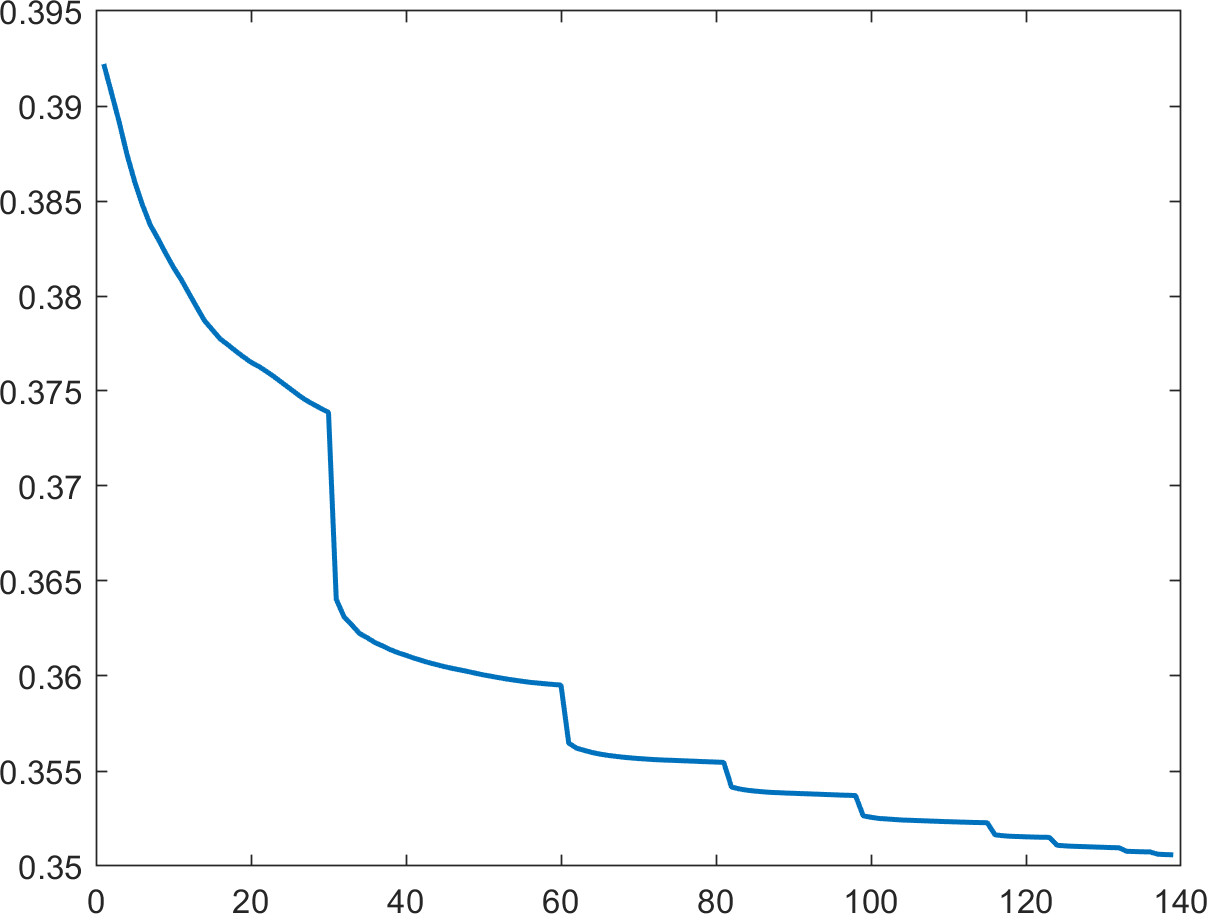}\hskip.35in
	\includegraphics[scale=0.35]{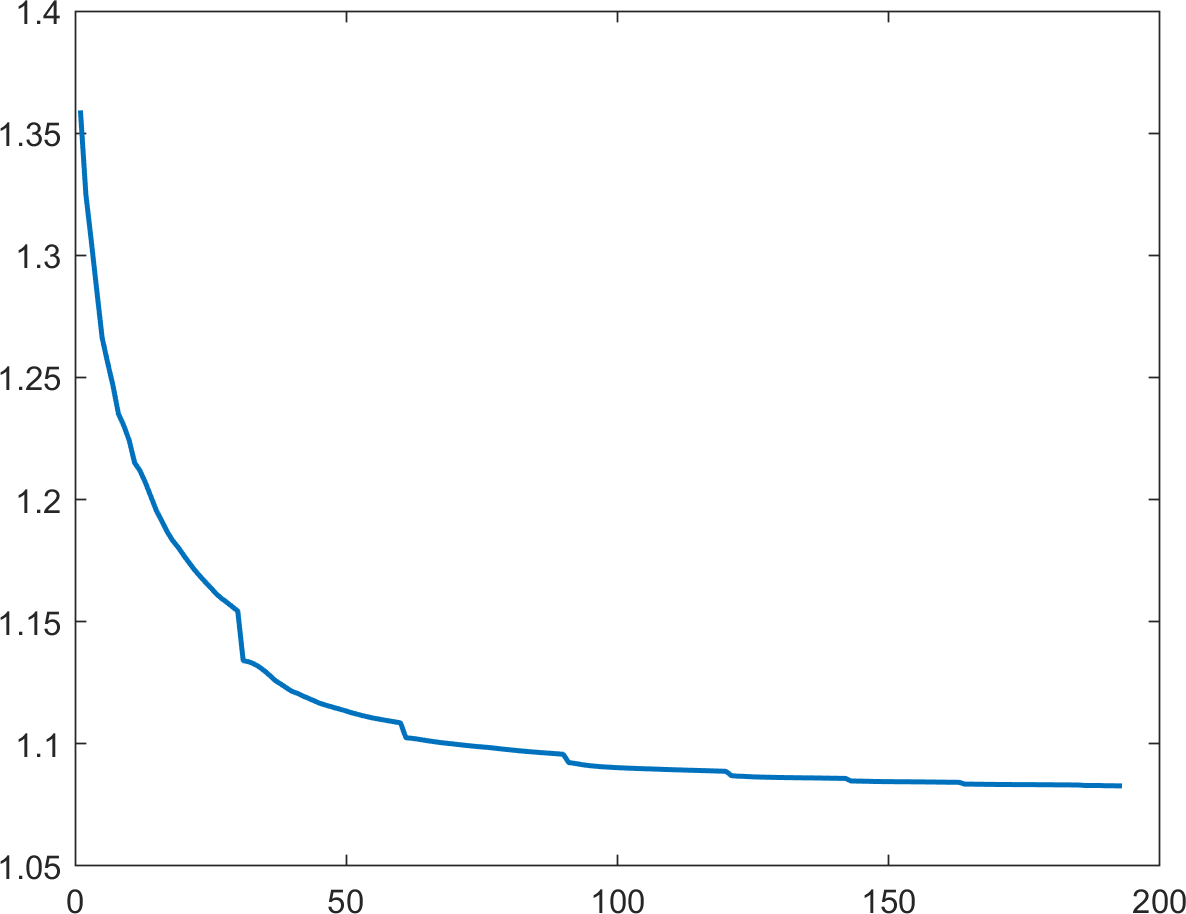}
	\caption{Examples of geodesics w.r.t. to the $(1,1,0.1,0)$ metric in the space of shapes $\Imm(S^2, \RR^3)/\Diff_+(S^2)$, where we choose  a resolution of $50 \times 99$, a maximal degree of spherical harmonics $\degree = \degbar = 7$ and  $13$ timesteps, i.e., we search in an approximately $2205$ dimensional space. The corresponding energy evolution for each example is shown on the bottom from left to right.}
	\label{fig.geodesic.energy.shape}
\end{figure*}

\begin{rem}
	The results in Table~\ref{table.time.multiresResults} suggest that our methods are well-suited for multiresolution methods, i.e., to solve the geodesic matching problem first on a coarser resolution (in both time, space, and degree of spherical harmonics) and then use an upsampled version of the previously obtained solution as initial guess for solving a high resolution version of the matching problem. Our numerical framework allows for these approaches in all available parameters and , in all our experiments this procedure seems to allow for as moderate improvements in the speed of the optimization. See Figure~\ref{fig.geodesic.multires} for an example of a multi-resolution geodesic.
\end{rem}

\subsection{Comparison to the SRNF-framework}
Finally, we aim to compare the results obtained with our method to the results using the inversion of linear paths in the SRNF-space. The SRNF metric corresponds to the split metric \eqref{eq.metric_abcd} with constant $(0, 1/2, 1, 0)$, see Appendix~\ref{appendix.A}. To demonstrate this correspondence, we consider 4 pairs of boundary surfaces. We calculated the length of the linear path between each pair of surfaces under the split  $(0, 1/2, 1, 0)$ and the length of the image of the linear path under the SRNF framework. The relative errors between the lengths for different time step sizes are shown in Table~\ref{table.isometry} and demonstrate that these two metrics indeed coincide.

\begin{table} [ht]
	\begin{adjustbox}{width=\columnwidth,center}
		\begin{tabular}{ | c | c | c | c | c|}
			\hline
			$\raisebox{.15cm}{\phantom{M}}\raisebox{-.15cm}{\phantom{M}}$
			Boundary Surfaces & Tstps & $L_l$ & $L_{L^2}$ & Relative Error \\ \hline
			{\multirow{3}{*}{\vspace{-.5cm}\hspace{-.3cm}\includegraphics[scale=.13]{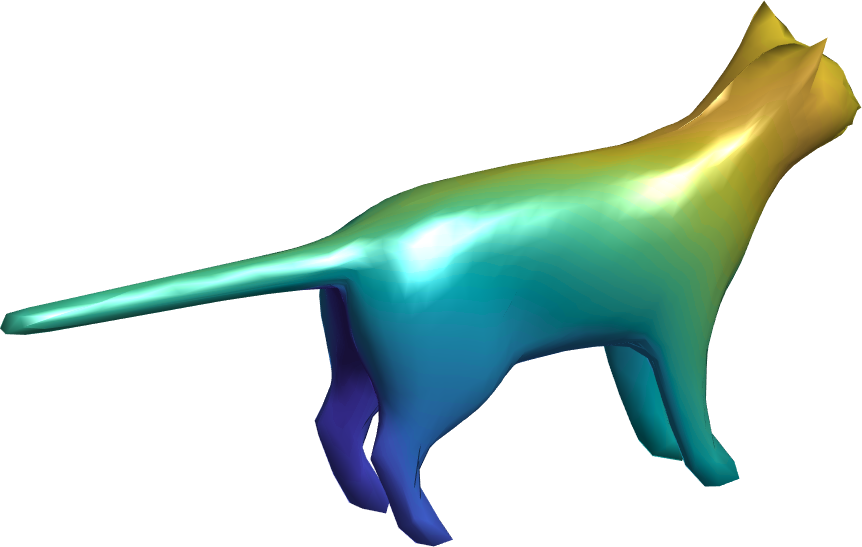}\quad\includegraphics[scale=.13]{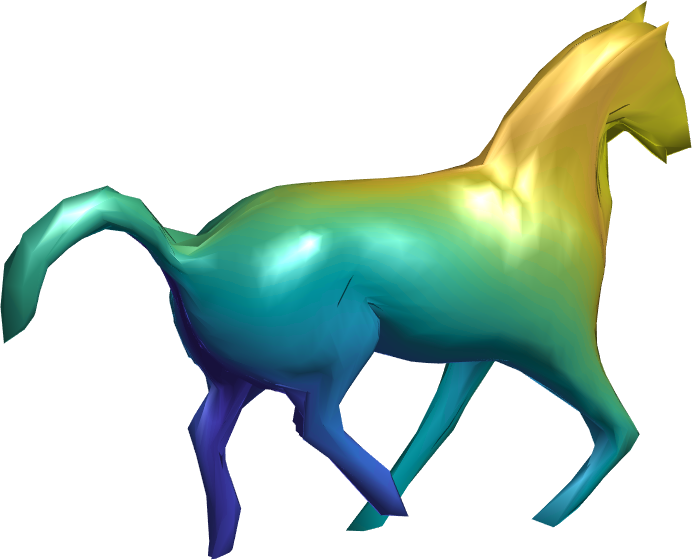}}}
			& $\raisebox{.15cm}{\phantom{M}}\raisebox{-.15cm}{\phantom{M}}$
			13 & 0.8917 & 0.8872 & 0.00502 \\ \cline{2-5}
			& $\raisebox{.15cm}{\phantom{M}}\raisebox{-.15cm}{\phantom{M}}$
			20 & 0.8952 & 0.8908 & 0.00494 \\ \cline{2-5}
			& $\raisebox{.15cm}{\phantom{M}}\raisebox{-.15cm}{\phantom{M}}$
			99 & 0.8952 & 0.8952 & 0.00002\\ \hline
			{\multirow{3}{*}{\includegraphics[scale=.2]{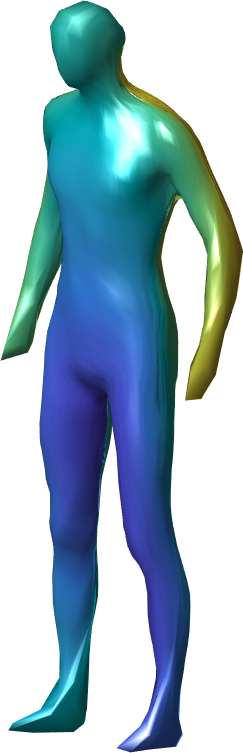}\,\qquad\includegraphics[scale=.2]{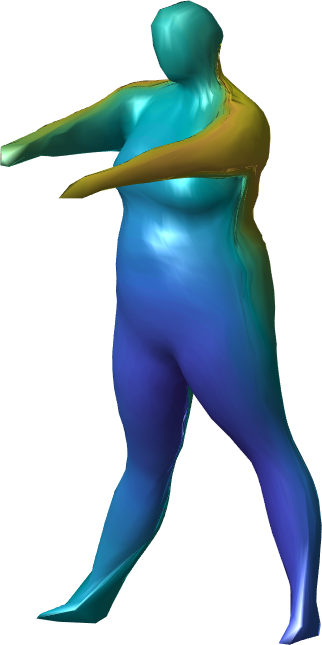}}}
			& $\raisebox{.25cm}{\phantom{M}}\raisebox{-.15cm}{\phantom{M}}$
			13 & 0.7384 & 0.7350 & 0.00456 \\ \cline{2-5}
			& $\raisebox{.25cm}{\phantom{M}}\raisebox{-.15cm}{\phantom{M}}$
			20 & 0.7372 & 0.7359 & 0.00169 \\ \cline{2-5}
			& $\raisebox{.25cm}{\phantom{M}}\raisebox{-.25cm}{\phantom{M}}$
			99 & 0.7371 & 0.7367 & 0.00053\\ \hline
			{\multirow{3}{*}{\vspace{-.5cm}\includegraphics[scale=.13]{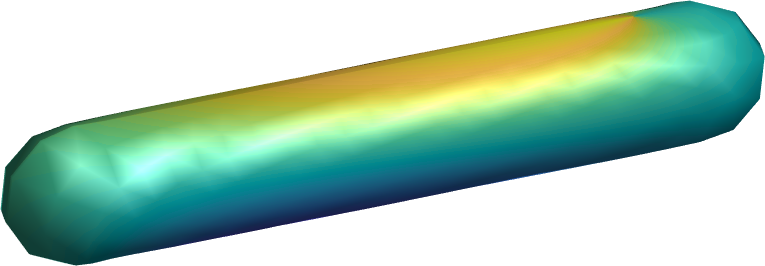}\qquad\includegraphics[scale=.13]{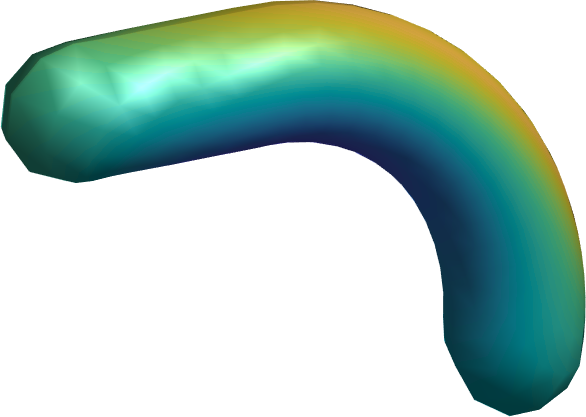}}}
			& $\raisebox{.15cm}{\phantom{M}}\raisebox{-.15cm}{\phantom{M}}$
			13 & 0.6722 & 0.6717 & 0.00072 \\ \cline{2-5}
			& $\raisebox{.15cm}{\phantom{M}}\raisebox{-.15cm}{\phantom{M}}$
			20 & 0.6722 & 0.6720 & 0.00030\\ \cline{2-5}
			& $\raisebox{.15cm}{\phantom{M}}\raisebox{-.15cm}{\phantom{M}}$
			99 & 0.6723 & 0.6723 & 0.00002 \\ \hline
			{\multirow{3}{*}{\vspace{-.5cm}\includegraphics[scale=.13]{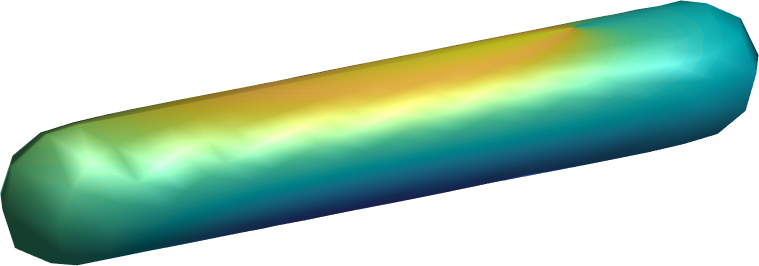}\qquad\includegraphics[scale=.13]{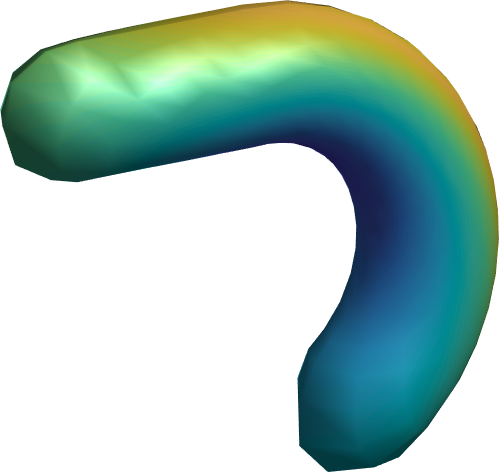}}}
			& $\raisebox{.15cm}{\phantom{M}}\raisebox{-.15cm}{\phantom{M}}$
			13 & 0.9875 & 0.9853 & 0.00226 \\ \cline{2-5}
			& $\raisebox{.15cm}{\phantom{M}}\raisebox{-.15cm}{\phantom{M}}$
			20 & 0.9874 & 0.9866 & 0.00084\\ \cline{2-5}
			& $\raisebox{.15cm}{\phantom{M}}\raisebox{-.15cm}{\phantom{M}}$
			99 & 0.9875 & 0.9875 & 0.00004\\ \hline
		\end{tabular}\\
	\end{adjustbox}
	\caption{Comparisons between the lengths of linear paths with respect to the split $(0,\frac12, 1,0)$ metric and the lengths of the SRNF representations of the linear paths with respect to the $L^2$ metric. $L_l$: the length of linear path; $L_{L_2}$: the length of the SRNF representation of the linear path with respect to the $L^2$ metric.}
	\label{table.isometry}
\end{table}

\begin{figure*}[ht]
	\centering
		\hspace{-.2in}\includegraphics[scale=0.45]{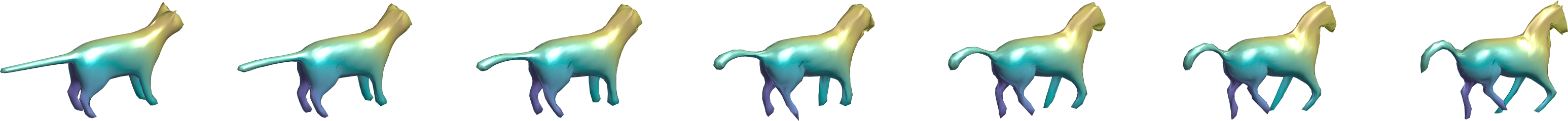}\\
		\hspace{-.2in}\includegraphics[scale=0.45]{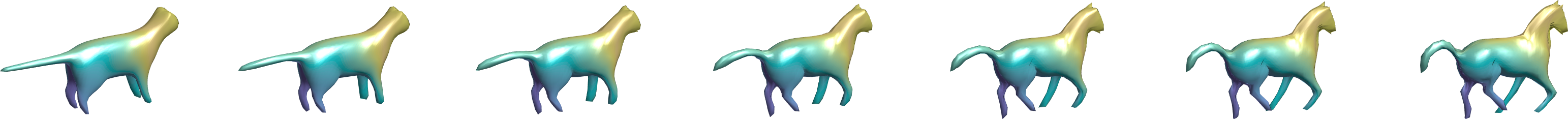}
		\vskip.1in
		\includegraphics[scale=0.45]{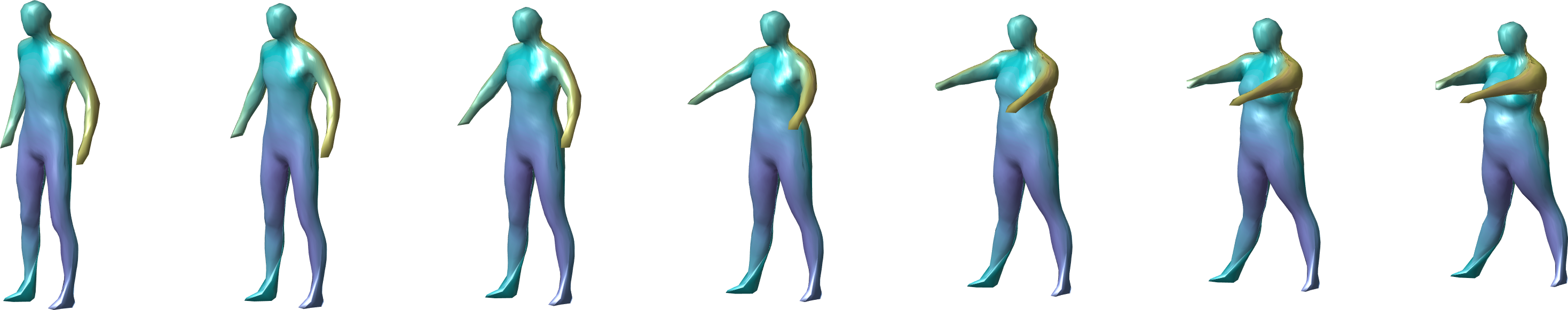}
		\includegraphics[scale=0.45]{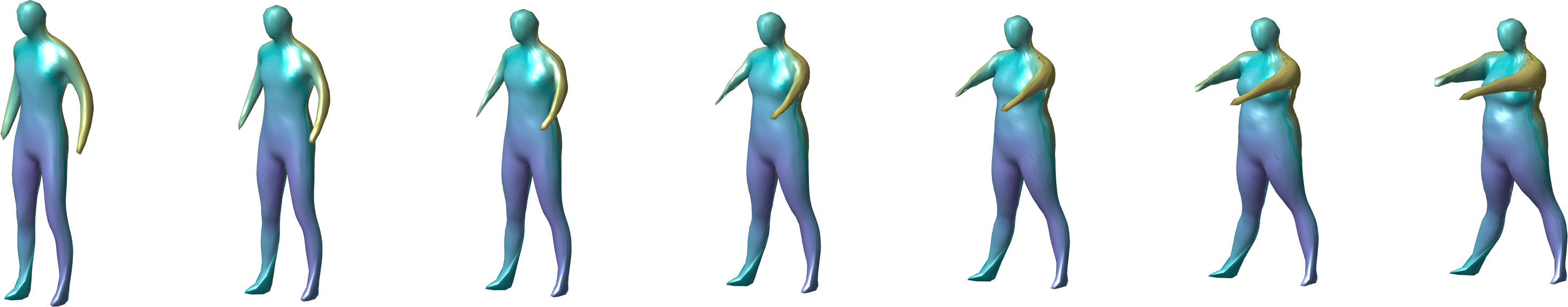}

	\caption{Comparisons of geodesics with respect to the  split $(0,\frac12,1,0)$ metric and the approximated inversions of straight lines under the SRNF framework. Row $1,3$: the approximated inversions under the SRNF framework; Row $2,4$: geodesics under the split $(0,\frac12,1,0)$ metric in the space of parametrized surfaces.}
	\label{fig.SurfacesComparison}
\end{figure*}

\begin{figure*}[ht]
	\centering
	\includegraphics[scale=0.45]{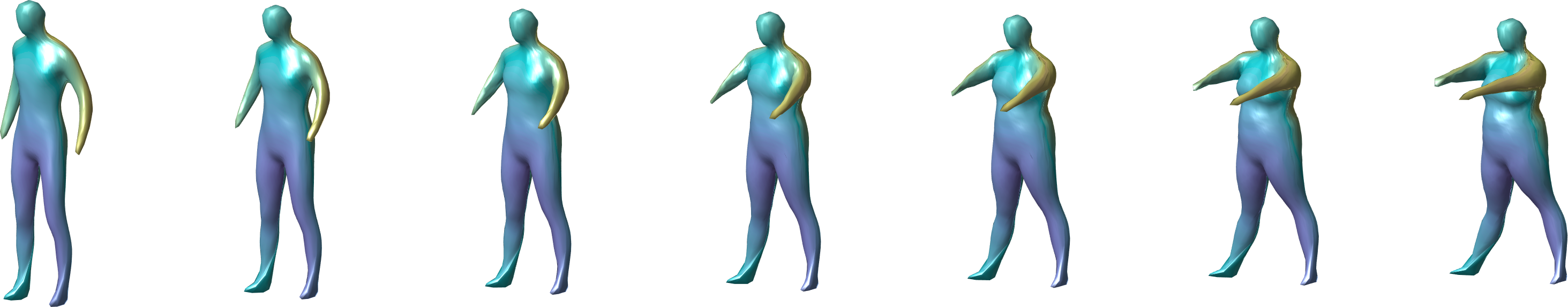} 
	\includegraphics[scale=0.45]{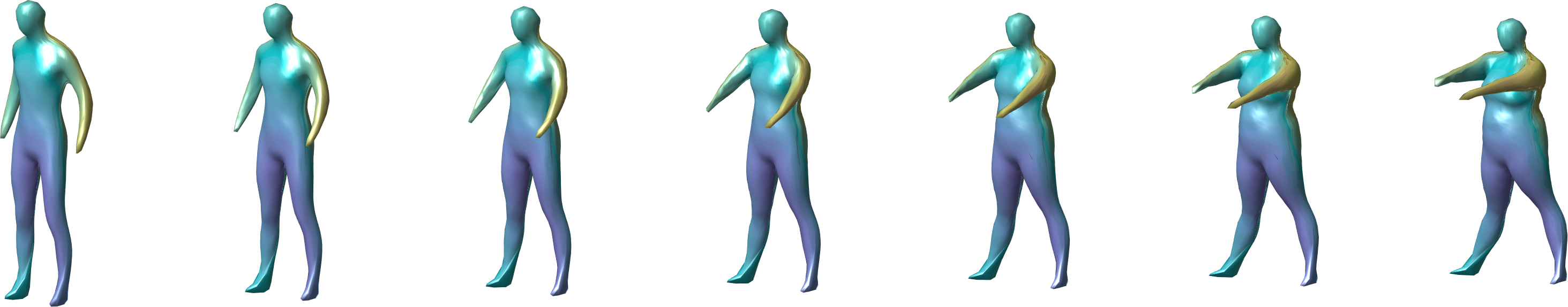}
	\caption{Geodesics between two human body surfaces in the space of unparametrized surfaces $\Imm(S^2,\RR^3)/\Diff_+(S^2)$ with respect to two different choices of coefficients  $(0,1,1,0)$ (top) and $(1,1,0.1,0)$ (bottom).
	In particular in the deformation of the arms one can observe the influence of the constants.}
	\label{fig.humanGeoDiffCoe}
\end{figure*}

Since the image of the SRNF map is not convex in $L^2$, the linear interpolation between two SRNFs may not have a preimage under the SRNF map. Also, even for functions that are in the image of SRNF map, the inverse does not have an analytic expression; in fact, such an expression does not exist in general, since the SRNF map is not injective. As a way to overcome this difficulty Laga et al. \cite{laga2017numerical} introduced a numerical method to calculate an approximated inversion of any path between two given SRNFs. In practice this has been used to approximate the geodesic by inverting the linear path between the given SRNFs. We want to remark here that the algorithm of \cite{laga2017numerical} could also be used to invert a geodesic in the image of the SRNF map. However, calculating geodesics in the image of the SRNF map is a non-trivial process, which to the best of our knowledge has not yet been attempted.  We would expect that this procedure would lead to minimizers that recover the minimizers obtained in the present framework. In Figure~\ref{fig.SurfacesComparison}, we consider two pairs of surfaces and calculate the geodesic between each pair of the boundary surfaces under the split $(0,1/2,1,0)$ metric with $\degree = \degbar = 7,  T = 13$ in the space of unparametrized surfaces $\Imm(S^2,\RR^3)/\Diff_+(S^2)$. The comparisons of these geodesics with the approximated inversions of the linear paths between the boundary surfaces are shown in Figure~\ref{fig.SurfacesComparison}. One can see that in the last row for the geodesic between the human body surfaces, the arms are shrinking at the beginning and then stretching, which maybe not a desired deformation for some applications. However, by adjusting the coefficients of our metric we could obtain geodesics with the natural behavior, see Figure~\ref{fig.humanGeoDiffCoe} for geodesics with respect to different choices of coefficients.

In Table~\ref{table.comparison} we compare the lengths of geodesics for four pairs of surfaces in the space of parametrized surfaces $\Imm(S^2,\RR^3)$, the lengths of the approximated inversions (with 7 time steps) under the split $(0,1/2,1,0)$ metric and the $L^2$ differences between the SRNFs of the boundary surfaces.  One can see from the table that for each pair of surfaces, the length of the geodesic is much closer to the $L^2$ difference than the length of the approximated inversion of the straight line between the SRNFs of the boundary surfaces. Note that the $L^2$-difference is a lower bound for the geodesic distance that will, in general, be strictly smaller then the true geodesic distance, as the image of the SNRF-represntation is not a totally geodesic (open) subspace of the space of all $L^2$-functions.

\begin{table}[ht]
	\begin{adjustbox}{width=\columnwidth,center}
		\begin{tabular}{ | c | c | c | c |c | c|}
			\hline
			$\raisebox{.15cm}{\phantom{M}}\raisebox{-.15cm}{\phantom{M}}$
			Boundary Surfaces & $L_l$ & $L_g$ & $L_i (7\text{stps})$ & $L^2$ Diff \\ \hline
			$\raisebox{.7cm}{\phantom{M}}\raisebox{-.3cm}{\phantom{M}}$
			\vspace{-0.1cm}\hspace{-.8cm}\includegraphics[scale=.13]{Figures/cat.png}\quad\includegraphics[scale=.13]{Figures/horse.png} & 0.8932 & 0.7948 & 1.1442 & 0.6130 \\ \hline
			$\raisebox{1.7cm}{\phantom{M}}\raisebox{-.3cm}{\phantom{M}}$
			\hspace{-.5cm}\includegraphics[scale=.2]{Figures/human1.png}\,\qquad\includegraphics[scale=.2]{Figures/human2.png} & 0.7380 & 0.7171 & 0.7919 & 0.6543 \\ \hline
			$\raisebox{.7cm}{\phantom{M}}\raisebox{-.3cm}{\phantom{M}}$
			\hspace{-.5cm}\includegraphics[scale=.13]{Figures/cylinder11.png}\qquad\includegraphics[scale=.13]{Figures/cylinder12.png} & 0.6723 & 0.5985 & 0.8393 & 0.5938\\ \hline
			$\raisebox{.7cm}{\phantom{M}}\raisebox{-.3cm}{\phantom{M}}$
			\hspace{-.5cm}\includegraphics[scale=.13]{Figures/cylinder21.png}\qquad\includegraphics[scale=.13]{Figures/cylinder22.png} & 0.9875 & 0.7973 & 1.2159 &  0.7786 \\ \hline
		\end{tabular}
	\end{adjustbox}
	\caption{The lengths of deformations with respect to the $(0,\frac12,1,0)$ metric between boundary surfaces with the maximal spherical harmonic degree of $7$ and time step size of $25$. $L_l$: the length of the linear path between boundary surfaces;  $L_g$: the length of geodesic as calculate in our numerical framework; $L_i$: the length of approximated inversion from SRNF straight line; $L^2$-Diff: the $L^2$ difference between the SRNFs of these boundary surfaces.}
	\label{table.comparison}
\end{table}

\begin{appendices}
\section{The geodesic equation}\label{Appendix:geodesic_equation}

In the following we give the geodesic equation on the space of immersions $\Imm(M, \RR^3)$ with respect to the pullback of the metric \eqref{eq.metric} on the space of $1$-forms.  In this Appendix, we will assume that the domain $M$ is a compact orientable surface without boundary, because we will need to use the Hodge decomposition. We will view $\alpha \in \Omega^1_+(M,\mathbb{R}^3)$ as a vector-valued $1$-form with components $( \alpha^1, \alpha^2,\alpha^3)$, where each $\alpha^i$ is a $1$-form on $M$ in the usual sense. Then the metric \eqref{eq.metric} can be rewritten as

\begin{align*}
G_{\alpha}(\xi, \xi) &= \int_M \tr(\xi \Lambda_{\alpha}\xi^T) \, \varphi_{\alpha} \, \mu  \\
&= \sum_{i=1}^3 \int_M \langle \xi_x^i, \Lambda_{\alpha} \xi_x^i\rangle \, \varphi_{\alpha} \, \mu
\end{align*}
where $\xi = (\xi^1, \xi^2,\xi^3)\in T_{\alpha}\Omega^1_+(M,\mathbb{R}^3)$, $\Lambda_{\alpha} = (\alpha^T\alpha)^{-1}$ is the induced Riemannian metric on $1$-forms on $M$, and $\varphi_{\alpha} = \sqrt{\det{(\alpha^T\alpha)}}$ is the induced volume form on $M$. As such all computations can be done one component at a time.

If $F = (f^1, f^2, f^3)$ is a vector-valued function with each $f^i\colon M\to \mathbb{R}$ real-valued, then $\beta = dF$ is a vector-valued $1$-form with $\beta^i = df^i$. The Hodge decomposition tells us that every $1$-form $\xi$ may be written as
$$ \xi = df + \gamma,$$
where $\delta \gamma = 0$ and $\delta:\Omega^1(M, \RR)\to C^{\infty}(M)$ is the codifferential operator.

The space $\Imm(M,\RR^3)$ is formally a submanifold of $\Omega^1_+(M,\mathbb{R}^3)$, and thus by general submanifold geometry we know that the geodesic equation on $\Imm(M,\mathbb{R}^3)$ will be given by
\begin{align}\label{submfdgeod}
\frac{D}{dt} \frac{d}{dt} \alpha = \gamma, \qquad \alpha = d\Phi, \qquad \delta \gamma = 0.
\end{align}
Since $\delta \gamma = 0$, we know that $\star\gamma$ is an exact form, where $\star$ denotes the Hodge star operator. Then there is a function $p$, unique up to a constant, such that $dp = \star\gamma$. We obtain
\begin{align}
	\Delta p = \delta dp = \delta\!\star\!\gamma = \star d\left( \frac{D}{dt} \frac{d}{dt} \alpha\right).
\end{align}
In coordinates $(u,v)$ on $M$ the operator $\star d$ is given by:
$$ \star d (f \, du + g \, dv) = \frac{g_u - f_v}{\varphi}.$$

From the geodesic equation on $\Omega^1_+(M,\mathbb{R}^3)$ with respect to the metric \eqref{eq.metric} in our previous paper \cite{bauer2018OneForms}, we know the covariant derivative is given by
\begin{align}
\frac{D}{dt} \frac{d\alpha}{dt} = \alpha_{tt} - \alpha_t(\alpha^T\alpha)^{-1} \alpha^T_t \alpha - \alpha_t \alpha^+\alpha_t + (\alpha_t\alpha^+)^T\alpha_t \\- \tfrac{1}{2} \tr(\alpha_t(\alpha^T\alpha)^{-1}\alpha^T_t)\alpha + \tr(\alpha_t\alpha^+) \alpha_t.
\end{align}
Since $d\alpha_{tt} = 0$, we obtain
\begin{multline*}
\Delta p = \star d\Big(-\alpha_t(\alpha^T\alpha)^{-1} \alpha^T_t \alpha\Big) - \alpha_t \alpha^+\alpha_t + (\alpha_t\alpha^+)^T\alpha_t \\
- \tfrac{1}{2} \tr(\alpha_t(\alpha^T\alpha)^{-1}\alpha^T_t)\alpha + \tr(\alpha_t\alpha^+) \alpha_t\Big).
\end{multline*}
Let $L = \alpha_t \alpha^+$. Then $\Phi$ is a geodesic on $\Imm(M, \RR^3)$ if and only if we have
\begin{align}
	\star \Delta p = d\Omega\wedge d\Phi
\end{align}
where $\Omega = LL^T + L^2 - L^T L +\frac12\tr(L^TL) - \tr(L)L$.
Here we emphasize that $p$ and $\Phi$ are actually vector-valued functions, so these computations are done componentwise for each $i\in \{1,2,3\}$. In other words, we have 
$$ \star \Delta p_i = \sum_{j=1}^3 d\Omega_{ij} \wedge d\Phi_j, \qquad i\in \{1,2,3\}.$$

\section{Proofs}\label{appendix.A}
\begin{proof}[Proof of Theorem~\ref{thm:correspondences}]
	In the following we prove the correspondence between our split metric on the space $\Omega_+^1(M,\RR^3)$ and the SRNF metric on the space of surfaces. Using the point-wise property of our metric we will focus on the corresponding split metric on the matrix space $M_+(3,2)$. For $a\in M_+(3,2)$ and $v\in T_aM_+(3,2)$, we decompose $v$ into four parts
	\begin{align}
	v = v_m + \frac12\tr(a^+v)a + v^{\perp} + v_0,
	\end{align}
	where
	\begin{align}
	v_m &= \frac12a(a^Ta)^{-1}(a^Tv + v^Ta) - \frac12\tr(a^+v)a\\
	v^{\perp} &= v - a(a^Ta)^{-1}a^Tv\\
	v_0 &= \frac12a(a^Ta)^{-1}(a^Tv - v^Ta).
	\end{align}
	The corresponding split metric on $M_+(3,2)$ is then of the form:
	\begin{multline}\label{eq.splitmetric.matrix}
	G^{\mathfrak{a},\mathfrak{b},\mathfrak{c},\mathfrak{d}}_{a}(v, v)\\
	=  \mathfrak{a}\langle v_m, v_m\rangle_a + \mathfrak{b}\left\langle\frac12\tr(a^+v)a, \frac12\tr(a^+v)a\right\rangle_a \\
	\qquad + \mathfrak{c}\langle v^{\perp}, v^{\perp}\rangle_a + \mathfrak{d}\langle v_0, v_0\rangle_a,
	\end{multline}
	Now consider the projection $\pi\colon M_+(3,2)\to\on{Sym}_+(2), \ a\mapsto a^Ta$. This projection is a Riemannian submersion, where $M_+(3,2)$ carries the metric \eqref{eq.splitmetric.matrix} with choices of constants $(1,1,1,1)$ and the space $\on{Sym}_+(2)$ is equipped with the following metric:
	\begin{align*}
	\langle h, k\rangle_g^{\on{Sym}} = \frac14\tr(g^{-1}hg^{-1}k)\sqrt{\det (g)}.
	\end{align*}
	The horizontal bundle with respect to the projection $\pi$ is given by
	\begin{align*}
	\mathcal{H}_a = \{u\in M(3,2)\,|\,ua^+\in\on{Sym}(n)\}
	\end{align*}
	and the differential $d\pi$ induces an isometry $$d\pi_a: \mathcal{H}_a \to T_{\pi(a)}\on{Sym}_+(m).$$ It is easy to check that $v_m$ and $\frac12\tr(a^+v)a$ are horizontal vectors.
	
	Let $g = \pi(a) =a^Ta$. By computation we have
	\begin{align*}
	\tr(a^+v) = \frac12\tr(g^{-1}d\pi_av)
	\end{align*}
	and
	\begin{align*}
	d\pi_a(v_m) &= a^Tv_m + v_m^Ta
	= a^Tv + v^Ta - \tr(a^+v)a^Ta\\
	&= d\pi_av - \frac12\tr\left(g^{-1}d\pi_a v\right)g.
	\end{align*}
	Therefore the first term in \eqref{eq.splitmetric.matrix} becomes
	\begin{align*}
	&\langle v_m, v_m\rangle_a\\
	&= \left\langle d\pi_a(v_m), d\pi_a(v_m)\right\rangle_{\pi(a)}^{\operatorname{Sym}}\\
	&=\left\langle d\pi_av - \frac12\tr\left(g^{-1}d\pi_a v\right)g, d\pi_av - \frac12\tr\left(g^{-1}d\pi_a v\right)g\right\rangle_g^{\operatorname{Sym}} \\
	&= \left\langle d\pi_av, d\pi_av\right\rangle_g^{\operatorname{Sym}} - \tr\left(g^{-1}d\pi_a v\right)\left\langle d\pi_av, g\right\rangle_g^{\operatorname{Sym}}\\
	&\qquad\qquad\qquad + \frac14\tr^2\left(g^{-1}d\pi_a v\right)\left\langle g, g\right\rangle_g^{\operatorname{Sym}} \\
	&= \frac14\tr\left( g^{-1}d\pi_avg^{-1}d\pi_av\right)\sqrt{\det(g)}\\ &\qquad\qquad\qquad-\frac18\tr^2\left(g^{-1}d\pi_a v\right)\sqrt{\det(g)}
	\end{align*}
	and the second term becomes
	\begin{align}
	\left\langle \frac12\tr(a^+v)a, \frac12\tr(a^+v)a\right\rangle_a &= \frac12\tr^2(a^+v) \sqrt{\det(a^Ta)}\\
	&= \frac18\tr^2(g^{-1}d\pi_av)\sqrt{\det (g)}.
	\end{align}
	For the third term in \eqref{eq.splitmetric.matrix}, we consider the corresponding unit normal map on the space of matrices given by
	\begin{align*}
	n: M_+(3,2) &\to \mathbb R^3\\
	a &\mapsto \frac{a_1\times a_2}{|a_1\times a_2|} = \frac{a_1\times a_2}{\sqrt{\det(a^Ta)}},
	\end{align*}
	where $a_1$ and $a_2$ are the first and the second columns of $a$, respectively.
	For any tangent vector $u = \begin{pmatrix}
	u_1 & u_2
	\end{pmatrix}$ at $a$, the differential of $n$ at $a$ is
	\begin{align*}
	dn_a(u) = \frac{u_1\times a_2+ a_1\times u_2 - (a_1\times a_2)\tr(a^+u)}{\sqrt{\det(a^Ta)}}.
	\end{align*}
	It is easy to check that $aa^+v$ is in the kernel of the differential $dn_a$, i.e.,
	\begin{align}
	dn_a(v) = dn_a(v^{\perp} + aa^+v) = dn_a(v^{\perp}).
	\end{align}
	Note that $\tr(a^+v^{\perp}) = 0$, $aa^+a_1 = a_1$ and $aa^+a_2 = a_2$. Using the following identity for three dimensional vectors $b, c, d, e$: $$(b\times c)\cdot (d\times e) = b^Tdc^Te - b^Tec^Td$$ and the formula for the inverse of $a^Ta$:
	\begin{align*}
	(a^Ta)^{-1} = \frac{1}{\det(a^Ta)}\begin{pmatrix}
	a_2^Ta_2 & -a_1^Ta_2\\
	a_1^Ta_2 & a_1^Ta_1
	\end{pmatrix},
	\end{align*}
	we have
	\begin{align}
	&\langle dn_av, dn_av\rangle_{\mathbb R^3} = \langle dn_av^{\perp}, dn_av^{\perp}\rangle_{\mathbb R^3}\\
	=& \frac{1}{\det(a^Ta)}\langle v^{\perp}_1\times a_2+ a_1\times v^{\perp}_2, v^{\perp}_1\times a_2+ a_1\times v^{\perp}_2\rangle_{\mathbb R^3}\\
	=& \frac{1}{\det(a^Ta)}\big[\left(v_1^Tv_1 - v_1^Taa^+v_1\right)a_2^Ta_2\\
	&\, - 2\left(v_1^Tv_2 - v_1^Taa^+v_2\right)a_1^Ta_2 + \left(v_2^Tv_2 - v_2^Taa^+v_2\right)a_1^Ta_1\big],
	\end{align}
	where $v^{\perp}_1, v^{\perp}_2$ are the first and the second columns of $v^{\perp}$ and $v_1, v_2$ are the first and the second columns of $v$, respectively. It follows that
	\begin{align*}
	&\left\langle v^{\perp}, v^{\perp}\right\rangle_a\sqrt{\det(a^Ta)} = \tr(v^{\perp}(a^Ta)^{-1}(v^{\perp})^T)\det(a^Ta)\\
	=& \left(\tr(v(a^Ta)^{-1}v^T) - \tr(aa^+v(a^Ta)^{-1}v^T)\right)\det(a^Ta)\\
	=& \left(\tr(v^Tv(a^Ta)^{-1}) - \tr(v^Taa^+v(a^Ta)^{-1})\right)\det(a^Ta)\\
	=& \tr\left(\begin{pmatrix}
	v_1^T(I - aa^+)v_1 & v_1^T(I - aa^+)v_2\\
	v_1^T(I - aa^+)v_2 & v_2^T(I - aa^+)v_2
	\end{pmatrix}\begin{pmatrix}
	a_2^Ta_2 & -a_1^Ta_2\\
	a_1^Ta_2 & a_1^Ta_1
	\end{pmatrix}\right)\\
	=& \left(v_1^Tv_1 - v_1^Taa^+v_1\right)a_2^Ta_2 - 2\left(v_1^Tv_2 - v_1^Taa^+v_2\right)a_1^Ta_2 \\
	&\qquad + \left(v_2^Tv_2 - v_2^Taa^+v_2\right)a_1^Ta_1\\
	=& \langle dn_av, dn_av\rangle_{\mathbb R^3}\det(a^Ta),
	\end{align*}
	that is,
	\begin{align*}
	\left\langle v^{\perp}, v^{\perp}\right\rangle_a = \langle dn_av, dn_av\rangle_{\mathbb R^3}\sqrt{\det(g)}.
	\end{align*}
	Therefore the split metric \eqref{eq.splitmetric.matrix} on $M_+(3,2)$ can be rewritten as
	\begin{align*}
	&G^{\mathfrak{a},\mathfrak{b},\mathfrak{c},\mathfrak{d}}_{a}(v, v)\\
	=& \mathfrak{a}\left(\frac14\tr\left( g^{-1}d\pi_avg^{-1}d\pi_av\right) -\frac18\tr^2\left(g^{-1}d\pi_a v\right)\right)\sqrt{\det(g)}\\
	&\quad + \frac {\mathfrak{b}}{8}\tr^2\left(g^{-1}d\pi_a v\right)\sqrt{\det(g)}+ \mathfrak{c}\langle dn_av, dn_av\rangle_{\mathbb R^3}\sqrt{\det(g)}\\
	&\hskip.35in + \mathfrak{d}\langle v_0, v_0\rangle_a.
	\end{align*}
	Now it is easy to see that the first three terms give rise to the formula of the full elastic metric on the space of surfaces and the SRNF metric corresponds to the split metric \eqref{eq.metric_abcd} with constants $(0 ,\frac12, 1, 0)$.
\end{proof}

\begin{proof}[Proof of Theorem~\ref{Proj_diff}]
	We first perform the computation in spherical coordinates $(\theta, \phi)\in [0, 2\pi]\times[0, \pi]$. Denote the usual spherical coordinate orthonormal basis by
	\begin{align*}
	e_1 &= \langle \sin{\phi} \cos{\theta}, \sin{\phi} \sin{\theta}, \cos{\phi}\rangle, \\
	e_2 &= \langle \cos{\phi} \cos{\theta}, \cos{\phi} \sin{\theta}, -\sin{\phi}\rangle, \\
	e_3 &= \langle -\sin{\theta}, \cos{\theta}, 0\rangle.
	\end{align*}
	We have the following formulas for the partial derivatives:
	\begin{alignat}{3}
	\partial_{\phi} e_1 &= e_2, \qquad \partial_{\phi} e_2 = -e_1, \qquad \partial_{\phi} e_3 = 0, \label{thetaderivs}\\
	\partial_{\theta} e_1 &= \sin{\phi} e_3, \qquad \partial_{\theta} e_2 = \cos{\phi} e_3, \\ \partial_{\theta} e_3 &= -\sin{\phi} e_1 - \cos{\phi} e_2.\label{phiderivs}
	\end{alignat}
	We also note that the covariant derivatives are given by
	\begin{align}\label{covderivs}
		\nabla_{e_2}e_2 &= 0,\quad &\nabla_{e_2}\ e_3 &= 0\\
		\nabla_{e_3}e_2 &= \cot{\phi}\, e_3,\quad &\nabla_{e_3}e_3 &= -\cot{\phi}\, e_2.
	\end{align}
	
	Write
	$$ U(\theta, \phi) = u(\theta, \phi) e_2(\theta, \phi) + v(\theta, \phi) e_3(\theta, \phi).$$
	For a real parameter $t$, we consider the following map $W\colon S^2 \to \mathbb{R}^3$ given in coordinates by
	\begin{align}
		W(\theta, \phi) &= e_1(\theta, \phi) + tU(\theta, \phi)\\
		&= e_1(\theta, \phi) + t u(\theta, \phi) e_2(\theta, \phi) + t v(\theta, \phi) e_3(\theta, \phi).
	\end{align}
	Then $\eta = W/\lvert W\rvert$.
	
	Note that in order for $\eta$ to be a diffeomorphism, we require that the Jacobian determinant be nonzero; it is given by
	\begin{align}
		\on{Jac}(\eta) = \frac{1}{\sin\phi}\left\lvert \frac{\partial \eta}{\partial \phi}\times \frac{\partial \eta}{\partial \theta}\right\vert.
	\end{align}
	Observe that
	\begin{align}
		\eta_{\phi} &= \frac{1}{\lvert W\rvert} \left( W_{\phi} - \frac{W\cdot W_{\phi}}{\lvert W\rvert^2}\, W\right) = \frac{1}{\lvert W\rvert} P_{W^{\perp}}(W_{\phi}),\\
		\eta_{\theta} &= \frac{1}{\lvert W\rvert} \left( W_{\theta} - \frac{W\cdot W_{\theta}}{\lvert W\rvert^2}\, W\right)= \frac{1}{\lvert W\rvert} P_{W^{\perp}}(W_{\theta}).
	\end{align}
	Since $\eta_{\phi}$ and $\eta_{\theta}$ are both perpendicular to $W$, we know that $\eta_{\phi}\times\eta_{\theta}$ is parallel to $W$; thus we obtain the formula
	\begin{align}
		\on{Jac}(\eta) &= \frac{1}{\sin\phi\lvert W\rvert^2}\big\lvert P_{W^{\perp}}(W_\phi)\times P_{W^{\perp}}(W_\theta)\big\rvert\\
		&= \frac{1}{\sin\phi\lvert W\rvert^3}\left\lvert W\cdot\big(P_{W^{\perp}}(W_\phi)\times P_{W^{\perp}}(W_\theta)\big)\right\rvert\\
		&= \frac{1}{\sin\phi\lvert W\rvert^3}\left\lvert W\cdot\big(W_\phi\times W_\theta\big)\right\rvert,
	\end{align}
	using the cyclic invariance of the scalar triple product and the fact that $W\times P_{W^{\perp}}(V) = W\times V$ for any vector $V$.
	
	Since $W = e_1+tU$ for the vector field $U = ue_2+ve_3$, it is straightforward to compute using \eqref{thetaderivs}-\eqref{covderivs} that
	\begin{align}
		W_{\phi} &= e_2+tU_{\phi} = e_2 +t\nabla_{e_2}U - tue_1,\\
		\frac{W_{\theta}}{\sin\phi} &= e_3+\frac{t}{\sin\phi}U_{\theta} = e_3 +t\nabla_{e_3}U - tve_1,
	\end{align}
	Let $a = u_{\phi}, b = v_{\phi}, c = \frac{u_{\theta}-v\cos{\phi}}{\sin{\phi}}, d=  \frac{v_{\theta}+u\cos{\phi}}{\sin{\phi}}$. We have by \eqref{covderivs} that
	\begin{align}
	\nabla_{e_2}U = ae_2+be_3,\qquad \nabla_{e_3}U = ce_2+de_3,
	\end{align}
which we abbreviate by 
	\begin{align}
		M :=\nabla U = \begin{pmatrix}
		a & b\\
		c & d
		\end{pmatrix}.
	\end{align}
	Thus the Jacobian is nonzero if and only if the following determinant is nonzero:
	\begin{align}\label{deter}
		D = \begin{vmatrix}
		1 & tu & tv\\
		-tu & 1+ ta & tb\\
		-tv & tc & 1+td
		\end{vmatrix}.
	\end{align}
	Then the determinant \eqref{deter} is given by
	\begin{align}
		D = \det(1+tM) + t^2\langle JU, (1+tM)JU\rangle,
	\end{align}
	where $J = \begin{pmatrix}
	0 & -1\\
	1 & 0
	\end{pmatrix}$.
	
	Let $\overline{M} = \tfrac{1}{2} (M+M^T)$ denote the symmetrization of $M$, and let $\lambda_1\le \lambda_2$ denote the real eigenvalues of $\overline{M}$. Then $\tr{M} = \tr{\overline{M}}$ and $\det{M} = \det{\overline{M}} + \tfrac{1}{4} (b-c)^2$, so that
	\begin{align}
		\det(1+tM) \geq \det(1+t\overline{M}) = (1+\lambda_1t)(1+\lambda_2t).
	\end{align}
	Since $J$ is a rotation, we have
	\begin{align}
		\langle JU, (1+tM)JU\rangle &= \langle JU, (1+t\overline{M})JU\rangle\\
		&\geq (1+\lambda_1t)\lvert JU\rvert^2\\
		&= (1+\lambda_1t)\lvert U\rvert^2.
	\end{align}
	Thus
	\begin{align}
		D\geq (1+\lambda_1t)(1+\lambda_2t+\lvert U\rvert^2t^2).
	\end{align}
	For sufficiently small $t$, we know $(1+\lambda_1t)$ is positive, and since $\lambda_1\leq\lambda_2$, we obtain
	\begin{align*}
		D\geq(1+\lambda_1t)^2
	\end{align*}
	Thus $1+\lambda_1t>0$ is a sufficient condition for positivity of $D$, and this happens as long as $\lvert t\rvert < \frac{1}{\lvert \lambda_1\rvert}$. It is easy to compute that
	$$
	\lambda_1 = \frac{a+d - \sqrt{(a-d)^2 + (b + c)^2}}{2}.
	$$
	In particular $a+d = \tr{(\nabla U)} = \on{div}{U}$, and by the divergence theorem, we know the integral of $a+d$ over $S^2$ is zero, and in particular $a+d$ is either identically zero or changes sign on $S^2$. Since $t$ is nonnegative we therefore are concerned about the most negative that $\lambda_1(x)$ can be:
	$$ 1 + \lambda_1(x) t \ge 1 + t\inf_{x\in S^2} \lambda_1(x) = 1 - t \sup_{p\in S^2}(-\lambda_1(x)) \ge 0,$$
	which is equivalent to
	$$ t < \frac{2}{\sup_{p\in S^2} -(a+d) + \sqrt{(a-d)^2 + (b+c)^2}}.$$
	This is clearly \eqref{tcondition}.
\end{proof}

\end{appendices}

\bibliographystyle{spmpsci}      %

\begin{thebibliography}{10}
\providecommand{\url}[1]{{#1}}
\providecommand{\urlprefix}{URL }
\expandafter\ifx\csname urlstyle\endcsname\relax
  \providecommand{\doi}[1]{DOI~\discretionary{}{}{}#1}\else
  \providecommand{\doi}{DOI~\discretionary{}{}{}\begingroup
  \urlstyle{rm}\Url}\fi

\bibitem{kinetsu1975Gauss}
Abe, K., Erbacher, J.: Isometric immersions with the same gauss map.
\newblock Mathematische Annalen \textbf{215}(3), 197--201 (1975)

\bibitem{brett2003Human}
Allen, B., Curless, B., Popovi{\'c}, Z.: The space of human body shapes:
  reconstruction and parameterization from range scans.
\newblock ACM Transactions on Graphics \textbf{22}(3), 587--594 (2003)

\bibitem{bauer2018relaxed}
Bauer, M., Bruveris, M., Charon, N., M{\o}ller-Andersen, J.: A relaxed approach
  for curve matching with elastic metrics.
\newblock To appear in ESAIM: COCV  (2018)

\bibitem{bauer2014overview}
Bauer, M., Bruveris, M., Michor, P.W.: Overview of the geometries of shape
  spaces and diffeomorphism groups.
\newblock Journal of Mathematical Imaging and Vision \textbf{50}(1-2), 60--97
  (2014)

\bibitem{bauer2011sobolev}
Bauer, M., Harms, P., Michor, P.W.: Sobolev metrics on shape space of surfaces.
\newblock Journal of Geometric Mechanics \textbf{3}(4), 389--438 (2011)

\bibitem{bauer2012sobolev}
Bauer, M., Harms, P., Michor, P.W.: Sobolev metrics on shape space, ii:
  Weighted sobolev metrics and almost local metrics.
\newblock Journal of Geometric Mechanics \textbf{4}(4), 365--383 (2012)

\bibitem{bauer2018OneForms}
Bauer, M., Klassen, E., Preston, S.C., Su, Z.: A diffeomorphism-invariant
  metric on the space of vector-valued one-forms.
\newblock arXiv:1812.10867  (2018)

\bibitem{bronstein2008numerical}
Bronstein, A.M., Bronstein, M.M., Kimmel, R.: Numerical geometry of non-rigid
  shapes.
\newblock Springer Science \& Business Media (2008)

\bibitem{CeEsSch2015}
Celledoni, E., Eslitzbichler, M., Schmeding, A.: Shape analysis on {Lie} groups
  with applications in computer animation.
\newblock The Journal of Geometric Mechanics \textbf{8}(3), 273--304 (2015)

\bibitem{cervera1991action}
Cervera, V., Mascaro, F., Michor, P.W.: The action of the diffeomorphism group
  on the space of immersions.
\newblock Differential Geometry and its Applications \textbf{1}(4), 391--401
  (1991)

\bibitem{grenandery1998Anatomy}
Grenander, U., Miller, M.I.: Computational anatomy: an emerging discipline.
\newblock Quarterly of Applied Mathematics \textbf{56}(4), 617--694 (1998)

\bibitem{hasler2009statistical}
Hasler, N., Stoll, C., Sunkel, M., Rosenhahn, B., Seidel, H.P.: A statistical
  model of human pose and body shape.
\newblock In: Computer graphics forum, vol.~28, pp. 337--346. Wiley Online
  Library (2009)

\bibitem{heeren2012time}
Heeren, B., Rumpf, M., Wardetzky, M., Wirth, B.: Time-discrete geodesics in the
  space of shells.
\newblock In: Computer Graphics Forum, vol.~31, pp. 1755--1764. Wiley Online
  Library (2012)

\bibitem{jermyn2017}
Jermyn, I., Kurtek, S., Laga, H., Srivastava, A.: Elastic shape analysis of
  three-dimensional objects.
\newblock Synthesis Lectures on Computer Vision \textbf{7}, 1--185 (2017)

\bibitem{jermyn2012SRNF}
Jermyn, I.H., Kurtek, S., Klassen, E., Srivastava, A.: Elastic shape matching
  of parameterized surfaces using square root normal fields.
\newblock Computer Vision – ECCV 2012 pp. 804--817 (2012)

\bibitem{kilian2007geometric}
Kilian, M., Mitra, N.J., Pottmann, H.: Geometric modeling in shape space.
\newblock In: ACM Transactions on Graphics (TOG), vol.~26, p.~64. ACM (2007)

\bibitem{klassenmichor2019}
Klassen, E., Michor, P.W.: On the non-invertibility of the square root normal
  function.
\newblock In preparation (preprint available on request).  (2019)

\bibitem{kurte2011Anatomical}
Kurtek, S., Klassen, E., Ding, Z., W.Jacobson, S., Jacobson, J.L., J.Avison,
  M., Srivastava, A.: Parameterization-invariant shape comparisons of
  anatomical surfaces.
\newblock IEEE Transactions on Medical Imaging \textbf{30}(3), 849--858 (2011)

\bibitem{kurtek2018simplifying}
Kurtek, S., Needham, T.: Simplifying transforms for general elastic metrics on
  the space of plane curves.
\newblock arXiv preprint arXiv:1803.10894  (2018)

\bibitem{kurtek2013landmark}
Kurtek, S., Srivastava, A., Klassen, E., Laga, H.: Landmark-guided elastic
  shape analysis of spherically-parameterized surfaces.
\newblock In: Computer graphics forum, vol.~32, pp. 429--438. Wiley Online
  Library (2013)

\bibitem{laga2017numerical}
Laga, H., Xie, Q., Jermyn, I.H., Srivastava, A.: Numerical inversion of srnf
  maps for elastic shape analysis of genus-zero surfaces.
\newblock IEEE Transactions on Pattern Analysis and Machine Intelligence
  \textbf{39}(12), 2451--2464 (2017)

\bibitem{mio2007Elastic}
Mio, W., Srivastava, A., Joshi, S.: On shape of plane elastic curves.
\newblock International Journal of Computer Vision \textbf{73}(3), 307--324
  (2007)

\bibitem{praun2003spherical}
Praun, E., Hoppe, H.: Spherical parametrization and remeshing.
\newblock In: ACM Transactions on Graphics (TOG), vol.~22, pp. 340--349. ACM
  (2003)

\bibitem{srivastava2011Shape}
Srivastava, A., Klassen, E., Joshi, S.H., Jermyn, I.H.: Shape analysis of
  elastic curves in euclidean spaces.
\newblock IEEE Transactions on Pattern Analysis and Machine Intelligence
  \textbf{33}(7), 1415--1428 (2011)

\bibitem{srivastava2016functional}
Srivastava, A., Klassen, E.P.: Functional and shape data analysis.
\newblock Springer (2016)

\bibitem{SuKuKlSr2014}
Su, J., Kurtek, S., Klassen, E., Srivastava, A.: Statistical analysis of
  trajectories on {Riemannian} manifolds: bird migration, hurricane tracking
  and video surveillance.
\newblock Ann. Appl. Stat. \textbf{8}(1), 530--552 (2014)

\bibitem{zhe2018Homogeneous}
Su, Z., Klassen, E., Bauer, M.: Comparing curves in homogeneous spaces.
\newblock Differential Geometry and its Applications \textbf{60}, 9--32 (2018)

\bibitem{tumpach2016gauge}
Tumpach, A.B.: Gauge invariance of degenerate riemannian metrics.
\newblock Notices of the AMS \textbf{63}(4) (2016)

\bibitem{tumpach2015gauge}
Tumpach, A.B., Drira, H., Daoudi, M., Srivastava, A.: Gauge invariant framework
  for shape analysis of surfaces.
\newblock IEEE transactions on pattern analysis and machine intelligence
  \textbf{38}(1), 46--59 (2015)

\bibitem{younes1998computable}
Younes, L.: Computable elastic distances between shapes.
\newblock SIAM Journal on Applied Mathematics \textbf{58}(2), 565--586 (1998)

\bibitem{younes2008metric}
Younes, L., Michor, P.W., Shah, J.M., Mumford, D.B.: A metric on shape space
  with explicit geodesics.
\newblock Rendiconti Lincei-Matematica e Applicazioni \textbf{19}(1), 25--57
  (2008)

\bibitem{ZhSuKlLeSr2015}
Zhang, Z., Su, J., Klassen, E., Le, H., Srivastava, A.: Video-based action
  recognition using rateinvariant analysis of covariance trajectories.
\newblock arXiv:1503.06699  (2015)

\end{thebibliography}

\end{document}